\documentclass[a4paper]{amsart}

\usepackage[english]{babel}

\usepackage[letterpaper,top=2cm,bottom=2cm,left=3cm,right=3cm, marginparwidth=1.75cm]{geometry}

\usepackage{amsmath}
\usepackage{amssymb}
\usepackage{amsthm}
\usepackage{mathtools}
\usepackage{graphicx}
\usepackage{tikz}
\usepackage{quiver}
\usepackage{verbatim}
\usetikzlibrary{cd, intersections, calc, arrows.meta}

\usepackage[dvipsnames]{xcolor}

\usepackage[all]{xy}
\definecolor{darkspringgreen}{rgb}{0.09, 0.45, 0.27}
\definecolor{dartmouthgreen}{rgb}{0.05, 0.5, 0.06}
\usepackage[colorlinks,final,hyperindex]{hyperref}
\usepackage[capitalize]{cleveref}
\hypersetup{citecolor=dartmouthgreen, urlcolor=blue, linkcolor=darkspringgreen}

\newtheoremstyle{results}
{8pt}
{8pt}
{\itshape}
{}
{\bfseries}
{}
{.5em}
{}
\theoremstyle{results}
\newtheorem{thm}{Theorem}[section]
\newtheorem{theorem}{Theorem}[section]
\newtheorem{lemma}[theorem]{Lemma}
\newtheorem{corollary}[theorem]{Corollary}
\newtheorem{prop}[theorem]{Proposition}

\newtheoremstyle{definitions}
{7pt}
{7pt}
{}
{}
{\bfseries}
{}
{.5em}
{}

\theoremstyle{definitions}
\newtheorem{definition}[thm]{Definition}
\newtheorem{defi}[thm]{Definition}

\newtheorem{example}[thm]{Example}

\newtheorem{remark}[thm]{Remark}

\newtheorem*{claim*}{Claim}
\newtheorem*{conjecture*}{Conjecture}

\crefname{thm}{theorem}{theorems}
\Crefname{thm}{Theorem}{Theorems}
\crefname{lemma}{lemma}{lemmas}
\Crefname{lemma}{Lemma}{Lemmas}
\crefname{proposition}{proposition}{propositions}
\Crefname{proposition}{Proposition}{Propositions}
\crefalias{proposition}{proposition}
\crefname{corollary}{corollary}{corollaries}
\Crefname{corollary}{Corollary}{Corollaries}
\crefname{definition}{definition}{definitions}
\Crefname{definition}{Definition}{Definitions}
\crefname{conjecture}{conjecture}{conjectures}
\Crefname{conjecture}{Conjecture}{Conjectures}
\crefname{example}{example}{examples}
\Crefname{example}{Example}{Examples}
\crefname{remark}{remark}{remarks}
\Crefname{remark}{Remark}{Remarks}

\DeclareMathOperator{\supp}{supp}

\newcommand{\C}{\operatorname{\mathcal{C}}}
\DeclareMathOperator{\Id}{Id}
\DeclareMathOperator{\id}{id}
\DeclareMathOperator{\op}{op}
\DeclareMathOperator{\Hom}{Hom}
\DeclareMathOperator{\End}{End}
\DeclareMathOperator{\PSh}{PSh}

\newcommand{\Cen}{\operatorname{\mathcal{Z}}}

\newcommand{\Top}{\operatorname{\mathsf{Top}}}

\newcommand{\Pos}{\operatorname{\mathsf{Pos}}}
\newcommand{\OM}{\operatorname{OM}}
\newcommand{\UM}{\operatorname{UM}}

\newcommand{\Po}{\operatorname{\mathsf{P}}}
\newcommand{\Qo}{\operatorname{\mathsf{Q}}}

\newcommand{\U}{\operatorname{U}}

\renewcommand{\L}{\mathcal{L}}

\newcommand{\rk}{\mathrm{rk}\,}
\newcommand{\si}{\mathsf{s}}

\newcommand{\N}{\mathcal{N}}
\newcommand{\B}{\mathrm{B}}

\newcommand{\Fl}{\operatorname{\mathcal{F}l}}

\newcommand{\Nmat}{\mathrm{N}}

\newcommand{\OS}{\mathrm{OS}}

\newcommand{\mat}{\mathrm{M}}
\newcommand{\om}{\mathcal{M}}
\newcommand{\Mat}{\om}
\newcommand{\umat}{\underline{\om}}

\newcommand{\La}{\mathcal{L}}

\newcommand{\sC}{\underline{C}}

\newcommand{\chain}{\mathcal{F}}
\newcommand{\Cofib}{\mathrm{Cofib}}
\newcommand{\Cyl}{\mathrm{Cyl}}
\newcommand{\Th}{\mathrm{Th}}
\newcommand{\gen}{\omega^{+}}

\newcommand{\Lie}{\mathsf{Lie}}
\newcommand{\Topes}{\mathsf{Topes}}

\newcommand{\Prop}{\mathsf{Prop}}
\newcommand{\WTopes}{\mathsf{WTopes}}
\renewcommand{\O}{\mathcal{O}}
\newcommand{\res}{\operatorname{res}}
\newcommand{\Sal}{\mathsf{Sal}}
\newcommand{\DF}{\mathsf{DF}}

\newcommand{\Del}{\mathsf{Del}}

\newcommand{\Pop}{\operatorname{\mathcal{P}}}

\newcommand{\Gr}{\operatorname{\Gamma}}

\newcommand{\Pro}{\operatorname{\mathbb{P}}}
\newcommand{\Sph}{\operatorname{\mathbb{S}}}
\newcommand{\A}{\mathcal{A}}
\newcommand{\W}{\mathcal{W}}
\renewcommand{\S}{\mathcal{S}}
\newcommand{\pr}{\mathrm{pr}}

\newcommand{\Real}{\mathbb{R}}
\newcommand{\Cmplx}{\operatorname{\mathbb{C}}}

\newcommand{\Field}{\operatorname{\mathbb{F}}}
\newcommand{\Z}{\mathbb{Z}}
\newcommand{\h}{\mathfrak{h}}
\newcommand{\Q}{\mathbb{Q}}

\title{Presheaf structures governed by oriented matroids \\  and
pseudohyperplane arrangements}

\author[B.\,Coron]{Basile Coron}
\address{B. Coron, LAREMA, 49100, Angers, France}\email{basile.coron@univ-angers.fr}

\author[D.\,Lyskov]{Denis Lyskov}
\address{Denis Lyskov: \newline
Université Paris Cité, IMJ-PRG, 75013, Paris,
France
\newline
Higher School of Economics, Faculty of Mathematics, 119048, Moscow, Russia}
\email{ddlyskov@gmail.com}

\date{}

\begin{document}

\begin{abstract}
We study higher complements of pseudohyperplane arrangements and the specialization maps between them, induced by degenerations. We develop a presheaf formalism on oriented matroids and weak maps, which model degenerations. Our main constructions are the posets of weighted and weightwise topes, which provide functorial combinatorial models for higher complements and their specialization morphisms. Their presheaf structure yields a deletion–contraction homotopy pushout, from which we derive a uniform combinatorial proof of the integral cohomology presentation for all $n\geq2$. This presentation is natural with respect to weak maps and identifies cohomology rings with higher Orlik–Solomon algebras extending the known formulas for linear arrangements to non-realizable oriented matroids. Finally, we equip Deligne groupoids with a presheaf structure and identify them functorially with the fundamental groupoids of the complexified complements.
\end{abstract}
\maketitle 
\setcounter{tocdepth}{1}
\tableofcontents

\section{Introduction}

The topology of arrangement complements lies at the crossroads between combinatorics and geometry. Orlik and Solomon~\cite{orlik1980combinatorics} have proved that the rational cohomology ring of a complex hyperplane arrangement complement can be described purely in terms of its underlying matroid. However, this result does not extend to the full homotopy type: Rybnikov~\cite{rybnikov2011fundamental} showed that there exist complex arrangements with the same underlying matroid but non-isomorphic fundamental groups of their complements. It was later shown~\cite{ArtalGuervilleViu2020} that this distinction persists even for complexifications of real arrangements. To retrieve the homotopy type of a complexified arrangement complement, the richer oriented data is essential. In a landmark paper, Salvetti~\cite{Salvetti1987} constructed a finite cell complex, completely determined by the combinatorial description of the faces and chambers of the underlying real arrangement, with the homotopy type of its complexified complement. \\

Oriented matroids extend chamber combinatorics beyond linear arrangements. By the topological representation theorem of Folkman and Lawrence~\cite{FolkmanLawrence1978}, they admit representations by pseudosphere arrangements, even when no linear realization exists. Gel'fand--Rybnikov and Bj\"orner--Ziegler~\cite{GelfandRybnikov1989,BjornerZiegler1992} extended the combinatorial study of complements to this setting; Deshpande~\cite{deshpande2016arrangements} gave a geometric interpretation using pseudohyperplane arrangements and their microbundle complements.

In this paper we study higher complements of pseudohyperplane arrangements and, in particular, the maps between them. For a central pseudohyperplane arrangement $\A=(H_e)_{e\in E}$ in $\Real^r$, its $n$-th complement is
\[
\U_n(\A)=(\Real^r)^n\setminus\bigcup_{e\in E}H_e^{\times n}.
\]
For $n=1$ this is the real complement; for $n=2$ and linear hyperplanes it is the usual complexified complement. For braid arrangements, for instance, these spaces are ordered configuration spaces of points in $\Real^n$. Higher Salvetti posets were introduced by Bj\"orner--Ziegler, Berger and Tamaki~\cite{BjornerZiegler1992,Berger2009HigherComplements,Tamaki2012} to model higher complements of pseudosphere arrangements. \\

An important motivation for this article, is to understand how higher complements vary in continuous families of arrangements. The main observation is that a degeneration of arrangements $\A\rightsquigarrow\A'$, given by a continuous family which is topologically trivial away from the special parameter, induces a specialization morphism
\[
\U_n(\A')\longrightarrow\U_n(\A)
\]
in the homotopy category of spaces. This map is obtained by comparing the special and generic fibres through the total complement of the family (Section~\ref{sec:degenerations}). We seek combinatorial models that describe each individual complements, as well as the specialization maps between them. \\

We organize such models as presheaves $\OM^{\mathrm{op}} \rightarrow \mathrm{Top}$ over the category $\OM$ whose objects are oriented matroids, and morphisms are given by a combinatorial counterpart of degeneration called weak maps. Gel'fand and MacPherson's work on Pontrjagin classes~\cite{GelfandMacPherson1992} initiated the use of oriented matroids in manifold and bundle theory, where weak specializations organise the combinatorics of varying tangent spaces. In this article, we use a generalization of the standard fixed-ground-set weak maps~\cite{anderson2001representing}, allowing reorientations and maps between different ground sets. We also develop a parallel theory for unoriented matroids, and classical weak maps between them. \\

The distinguished weak maps $\triangle_F\colon\om\to\om/F\oplus\om|_F$, associated to any flat $F$, reveal the operadic aspect of the theory. A presheaf $\mathcal O$ turns them into composition maps
\[
\mathcal O(\om/F)\otimes\mathcal O(\om|_F)\longrightarrow\mathcal O(\om),
\]
whose associativity follows from functoriality. On braid matroids these recover ordinary operadic compositions. More generally, these products are the oriented counterparts of the first author's $\mathfrak{LBS}$-operadic compositions, in particular those indexed by minimal building sets~\cite{coron2024matroidsfeynmancategorieskoszul}, and of his $\mathcal P$-operadic compositions indexed by intervals in geometric lattices~\cite{Coron2025OperadicKLS}. On graphic matroids they give the contraction products of the second author's contractads~\cite{Lyskov2024GeneralizationOperads}. The presheaf formalism extends this operadic structure to all weak maps.\\

The usual covector-wise maps between Salvetti posets give functoriality for strong maps between oriented matroids, but do not extend directly to arbitrary weak maps. Hence, higher Salvetti posets are insufficient for our purpose. To address this issue we introduce the posets of weighted topes $\Topes_n$ and weightwise topes $\WTopes_n$ associated to any matroid. Both constructions form $\OM$-presheaves, and for every pseudohyperplane arrangement $\A$ we prove
\[
|\Topes_n(\om(\A))|\simeq
|\WTopes_n(\om(\A))|\simeq\U_n(\A).
\]
We compare these new models with higher Salvetti posets and, when $n=2$, with the existing tope model of Delucchi and Falk~\cite{DelucchiFalk2017}. Proposition~\ref{prop:degeneration-compatibility} identifies the pullback induced by a degeneration with its specialization morphism. Thus these poset morphisms give explicit representatives for specialization morphisms, determined by the induced weak maps.\\

Beyond complement models, Section~\ref{sec:weak} includes other examples of $\OM$-presheaves, such as rainbow-circuit-free colourings~\cite{berczi2020rainbow}, circuit varieties~\cite{ClarkeGraceMohammadiMotwani2023}, holonomy Lie algebras~\cite{suciu2025resonance}, as well as higher Orlik--Solomon algebras.\\

The main application of our formalism is to establish the following deletion--contraction homotopy pushout of posets (Theorem~\ref{thm:deletion-contraction-pushout})

\begin{equation*}
\begin{tikzcd}[column sep=large]
\Topes_n(\om/e)\times\Topes_n(\om|_e)
  \arrow[r]\arrow[d]
  &\Topes_n(\om)\arrow[d]\\
\Topes_n(\om/e)\arrow[r]
  &\Topes_n(\om\setminus e)\arrow[ul, phantom,
        "\overset{h}{\ulcorner}"{description, pos=0.15}],
\end{tikzcd}
\end{equation*}
which relates higher complements of deletion and contraction at an element $e$, in a purely combinatorial setting. The geometric principle behind this homotopy pushout is familiar in arrangement theory: deleting a hyperplane amounts to filling in a neighbourhood of the complement of the restricted arrangement. At the combinatorial level, the presheaf structure is needed to even define the gluing maps of the square. It supplies the top horisontal arrow from the weak map $\om\to\om/e\oplus\om|_e$, and the lower horizontal arrow from $\om\setminus e\to\om/e$. \\

From this square, we derive the classical deletion-contraction sequence in integral cohomology 
\[
0 \to H^{\bullet}(\Topes_n(\om\setminus e))\rightarrow H^{\bullet}(\Topes_n(\om))\to H^{\bullet-(n-1)}(\Topes_n(\om/e))\to 0
\]
This sequence is then used to recover the classical integral cohomology presentation
\[
H^\bullet(\U_n(\A);\Z)\cong\OS_n^\bullet(\om(\A)),
\qquad n\geq2,
\]
for any pseudohyperplane arrangement. Here, $\OS^{\bullet}_n(\om)$ has generators in degree $n-1$. For even $n$ this is a regrading of the ordinary Orlik--Solomon algebra; for odd $n$ it is a regrading of Cordovil's algebra~\cite{Cordovil2002CommutativeAlgebraOrientedMatroids}. Theorem~\ref{thm::pres_H(Topes)} further proves that the identification above is an isomorphism of $\OM$-copresheaves. For $n=2$, the cohomology identification for arbitrary oriented matroids is due to Gel'fand and Rybnikov~\cite{GelfandRybnikov1989}; a proof in the more general setting of complex matroids appears in~\cite[Theorem~4.9]{Ziegler1993ComplexMatroid}. In higher dimensions our theorem extends the formula of de Longueville and Schultz for linear arrangements~\cite{delongueville2001cohomology} to non-realisable oriented matroids. Our new proof works uniformly for all $n\geq2$, deriving the generators, circuit relations and their functoriality from the presheaf structure and deletion--contraction. It uses only categorical and combinatorial arguments, which remain available when there is no linear or complex realization on which to define logarithmic differential forms. \\

As additional applications, we obtain generalizations of several classical results. We give a splitting of the stable homotopy type of higher complements, we show connectivity results and vanishing of higher Hurewicz homomorphisms. Finally, Section~\ref{sec:deligne} equips positive Deligne categories and Deligne groupoids with $\OM$-presheaf structures and proves $\Del\simeq\Pi_1(\Topes_2)$, extending Deligne's combinatorial approach~\cite{deligne1972immeubles}. A categorical circle action, compatible with weak maps, recovers the scalar circle action at the groupoid level in the realizable case. \\

\textbf{AI usage statement.} The project originated before the available AI tools were
strong enough to meaningfully contribute. In the later stages ChatGPT pro was
used for proofreading, discussions and literature searches.

\section{Oriented matroids and Pseudohyperplane arrangements} In this section, we provide the background and notation used throughout the paper. We recall the basics of unoriented matroids and oriented matroids and the representation of oriented matroids by pseudohyperplane/pseudosphere arrangements. We then recall higher arrangement complements, Salvetti posets, Orlik--Solomon and Cordovil algebras. 
\subsection{Matroids and hyperplane arrangements} We recall the basics of unoriented matroids, refering to~\cite{oxley2011matroid} for more details.

A matroid $\mat$ is a pair $(E(\mat),\rk_{\mat})$ consisting of a finite set $E(\mat)$ and a rank function $\rk_{\mat}\colon\Pop(E(\mat))\to \mathbb{N}$ satisfying the following axioms.
\begin{itemize}
    \item[(i)] $\rk_{\mat}(\varnothing)=0$
    \item[(ii)] $\rk_{\mat} A \leqslant \rk_{\mat}A\cup \{e\}\leqslant \rk_{\mat}A+1, \quad \forall S \subseteq E(\mat), e\in E(\mat)$.
    
    \item[(iii)] $\rk_{\mat} (A\cap B)+\rk_{\mat}(A\cup B)\leq \rk_{\mat} A+\rk_{\mat} B, \quad \forall A, B \subseteq E(\mat).$
\end{itemize}
Throughout this paper, we will always assume our matroids to be \textit{loopless}, that is, containing no element of rank zero. A subset $I\subseteq E(\mat)$ is \textit{independent} if $\rk_{\mat} I=|I|$; an independent set of maximal cardinality is called a \textit{basis}. Minimal dependent sets are called \textit{circuits}. Recall that there are axiomatic descriptions of matroids in terms of circuits and bases respectively.

A subset $F\subseteq E(\mat)$ is called a \textit{flat} if for any $x\not\in F$ we have $\rk_{\mat} F\cup\{x\}>\rk_{\mat} F$. We will denote by $\Fl(\mat)$ the poset of flats of $\mat$ ordered by inclusion. A matroid is \textit{simple} if every singleton is a rank $1$ flat, i.e., there are no elements $e\neq e'$ that are \textit{parallel}, meaning that $\rk_{\mat}\{e,e'\}=1$. Recall that there is a one-to-one correspondence
\[
\Fl(-)\colon \text{Simple Matroids}\overset{1:1}{\longrightarrow} \text{ Geometric Lattices}.
\] 
\begin{example}[Examples of matroids] Let us list several classes of examples of matroids. 
\begin{itemize}
    \item[(Uniform)] For integers $n,r$, the uniform matroid $\U_{r,n}$ is the matroid on the ground set $[n]=\{1,2,\cdots,n\}$ with the rank function
    \[
    \rk_{\U_{r,n}} (S)=\min (|S|,r)
    \]
    \item[(Graphic)] For a loopless graph $\Gr$ with vertex set $V_{\Gr}$ and edge set $E_{\Gr}$, the graphic matroid $\mat(\Gr)$ is the matroid on the ground set $E_{\Gr}$ whose circuits are exactly the cycles of $\Gr$.
    \item[(Realizable)] For a field $\Field$ and a central hyperplane arrangement $\A=\{H_1,\cdots,H_n\}\subset \Field^r$, the matroid $\mat(\A)$ has ground set $[n]$ with rank function
    \[
    \rk_{\mat(\A)}(S)=\mathrm{codim}_{\Field^r}(\cap_{H\in S} H)
    \] We provide some examples of realizable and non-realizable matroids in Figure~\ref{fig::fano_matroid} and Figure~\ref{fig::pappus}
\end{itemize}
\end{example}
We recall several important constructions on matroids:
\begin{itemize}
    \item[(Simplification)] For a matroid $\om$, its simplification $\si\mat$ is the unique simple matroid with the same intersection lattice $\Fl(\mat)=\Fl(\si\mat)$. Explicitly, elements of $\si\mat$ are rank-one flats of $\mat$ with the rank function induced from $\mat$.
    \item[(Direct Sum)] For a pair of matroids $\mat,\Nmat$, their \textit{direct sum} $\mat\oplus \Nmat$ is a matroid on the ground set $E(\mat)\cup E(\Nmat)$ and rank function $\rk_{\mat\oplus \Nmat}(S)=\rk_{\mat}(S\cap E(\mat))+\rk_{\Nmat}(S\cap E(\Nmat))$. This construction is determined by the isomorphism of intersection lattices
    \[
    \Fl(\mat\oplus\Nmat)\cong \Fl(\mat)\times\Fl(\Nmat)
    \]
    \item[(Restriction)] For a subset $S\subset E(\mat)$, the restriction matroid, denoted by $\mat|_S$, is the induced matroid on $S$. Sometimes, we will denote by $\mat\setminus S:=\mat|_{E(\mat)\setminus S}$ the deletion matroid
    \item[(Contraction)] For a flat $F$ of $\mat$, the contraction matroid $\mat/F$ is the matroid on the ground set $E(\mat)\setminus F$ with the rank function
    \[
    \rk_{\mat/F}(S)= \rk_{\mat} S\cup F-\rk F.
    \] In the realizable case, a contraction of $\mat(\A)$ is again realizable
    \[
    \mat(\A)/F=\mat(\A^F),
    \] where $\A^{F}=\{H_e\cap H_F\}_{e\in E\setminus F}\subset H_F=\cap_{f\in F} H_f$ is the \textit{restricted arrangement} 
    
    The definition of contraction does not require $F$ to be a flat, but this condition ensures that $\mat/F$ is loopless. 
\end{itemize}   

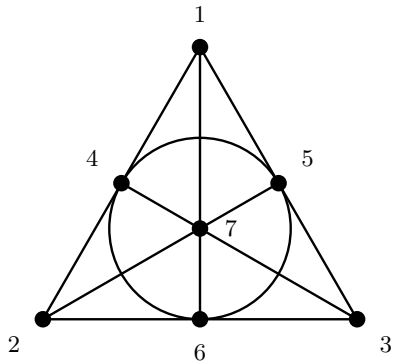
\begin{figure}[ht]
    \centering
    \begin{tikzpicture}[
    line width=0.9pt,
    point/.style={circle, fill=black, inner sep=2.2pt},
    every label/.style={font=\small},
    label distance=2pt, scale=0.8
]
    \coordinate (A) at ( 90:3);
    \coordinate (B) at (210:3);
    \coordinate (C) at (330:3);

    \coordinate (D) at ($(A)!0.5!(B)$);
    \coordinate (E) at ($(A)!0.5!(C)$);
    \coordinate (F) at ($(B)!0.5!(C)$);
    \coordinate (O) at (0,0);

    \draw (A) -- (B) -- (C) -- cycle;
    \draw (A) -- (F);
    \draw (B) -- (E);
    \draw (C) -- (D);
    \draw (O) circle[radius=1.5];

    \node[point, label= 90:{$1$}] at (A) {};
    \node[point, label=210:{$2$}] at (B) {};
    \node[point, label=330:{$3$}] at (C) {};
    \node[point, label=150:{$4$}] at (D) {};
    \node[point, label= 30:{$5$}] at (E) {};
    \node[point, label=270:{$6$}] at (F) {};
    \node[point, label=  0:{$7$}] at (O) {};
\end{tikzpicture}
    \caption{Fano matroid $\mathsf{F}_7$ on $7$ elements. Lines correspond to rank $2$ flats that contain $3$ elements. This matroid is representable only over fields of characteristic $2$}\label{fig::fano_matroid}
\end{figure}

\begin{figure}[ht]
    \centering
    \begin{tikzpicture}[
    x=1.2cm, y=1.2cm,
    line width=0.8pt,
    line cap=round,
    line join=round,
    every node/.style={font=\small}, scale=0.65
]
    \foreach \shift/\name/\nonpappus in {0/P/0,8/N/1} {
        \begin{scope}[xshift=\shift cm]

            \coordinate (1) at (0,3);
            \coordinate (2) at (2,3);
            \coordinate (3) at (5,3);
            \coordinate (4) at (0,0);
            \coordinate (5) at (2,0);
            \coordinate (6) at (5,0);

            \coordinate (7) at (1,1.5);
            \coordinate (9) at (3.5,1.5);

            \ifnum\nonpappus=0\relax
                \coordinate (8) at (2.5,1.5);
            \else
                \coordinate (8) at (2.15,1.71);
            \fi

            \draw (1) -- (3);  
            \draw (4) -- (6);  
            \draw (1) -- (5);  
            \draw (2) -- (4);  
            \draw (1) -- (6);  
            \draw (2) -- (6);  
            \draw (3) -- (5);  

            \ifnum\nonpappus=0\relax
                \draw (4) -- (3); 
                \draw (7) -- (9); 
            \else
                \draw (4)
                    .. controls (1.2,1.28) and ({43/15},2.28)
                    .. (3);
            \fi

            \foreach \i in {1,...,9}
                \fill (\i) circle[radius=2pt];

            \foreach \i in {1,2,3}
                \node[above=4pt] at (\i) {$\i$};

            \foreach \i in {4,5,6,7,8,9}
                \node[below=4pt] at (\i) {$\i$};

            \node[font=\large] at (2.5,-0.85)
                {$\mathsf{\name}_9$};

        \end{scope}
    }
\end{tikzpicture}
    \caption{Pappus $\mathsf{P}_9$ and Non-Pappus $\mathsf{N}_9$ matroids on $9$ elements. The first one is realizable over any field with at least $4$ elements. The second one isn't realizable over any field}\label{fig::pappus}
\end{figure}
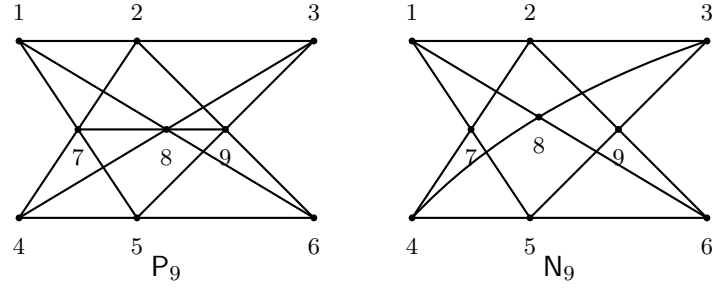

\subsection{Oriented matroids and real hyperplane arrangements}
In this section we review the notion of an oriented matroid and introduce the notation needed later. We refer to~\cite{bjorner1999oriented} for further details.

A \emph{signed set}, or \emph{sign vector}, on a finite set $E$ is a map $X:E\to\{+,-,0\}$. We write
\[
 X^{\pm}=\{e\in E\mid X(e)=\pm\},\qquad
 \underline{X}=\supp X=X^+\cup X^-.
\]
The zero sign vector is denoted by $0$, and $-X$ is obtained by reversing every nonzero sign. Identifying $+$ and $-$ with $1$ and $-1$, we define the pointwise product by $(X\cdot Y)(e)=X(e)Y(e)$. A signed set is \emph{constant} if its restriction to its support is constant. The \emph{composition} of two signed sets $X,Y$ is the signed set $X\circ Y$ defined by 
\[
 (X\circ Y)(e)=
 \begin{cases}
 X(e),&X(e)\ne0,\\
 Y(e),&X(e)=0.
 \end{cases}
\]
We write $X\preceq Y$ if $X(e)\in\{0,Y(e)\}$ for every $e\in E$, or equivalently $X^\pm\subseteq Y^\pm$. This partial order is also denoted by $\leq$. The \emph{separator} of $X$ and $Y$ is
\[
 S(X,Y)=\{e\in\underline{X}\cap\underline{Y}\mid X(e)=-Y(e)\}.
\]
We say that $X$ and $Y$ are \emph{orthogonal}, written $X\perp Y$, if their supports are disjoint or $X\cdot Y$ takes both signs on their common support. As with ordinary matroids, there are several axiom
systems for oriented matroids~\cite[Sec~3]{bjorner1999oriented}, we use only one of them.

\begin{definition}[Circuit axioms]\label{def:oriented-matroid-circuits}
An \emph{oriented matroid} is a pair $\om=(E,\C)$, where $E=E(\om)$ is a finite ground set and $\C\subseteq\{+,-,0\}^{E}$ satisfies:
\begin{itemize}
 \item[(C0)] $0\notin\C$.
 \item[(C1)] If $C\in\C$, then $-C\in\C$.
 \item[(C2)] If $C,D\in\C$ and $\underline{C}\subseteq\underline{D}$, then $C=\pm D$.
 \item[(C3)] If $C,D\in\C$, $C\ne-D$, and $e\in S(C,D)$, then there exists $Z\in\C$ such that
 \[
 Z^+\subseteq(C^+\cup D^+)\setminus\{e\},\qquad
 Z^-\subseteq(C^-\cup D^-)\setminus\{e\}.
 \]
\end{itemize}
The elements of $\C$ are the \emph{signed circuits} of $\om$.
\end{definition}

Their supports are the circuits of the \emph{underlying matroid} $\umat$. As in Section~2.1, we assume our oriented matroids to be \emph{loopless}, meaning that every circuit has support of cardinality at least two. Distinct elements $e,f$ are \emph{parallel}, respectively \emph{antiparallel}, if some circuit $C$ has support $\{e,f\}$ and $C(e)=-C(f)$, respectively $C(e)=C(f)$.

A \emph{covector} of $\om$ is a signed set orthogonal to every circuit. We denote the set of covectors, ordered by $\preceq$, by
\[
 \La(\om)=\{X\in\{+,-,0\}^{E}\mid X\perp C\text{ for all }C\in\C\}.
\]
$\La(\om)$ contains $0$ and is closed under sign reversal and composition. A \emph{tope} is a maximal covector; since $\om$ is loopless, these are precisely the covectors with support $E$. Their set is denoted by $\Topes(\om)$.

As for unoriented matroids, hyperplane arrangements provide interesting examples of oriented matroids. 
\begin{example}[realizable oriented matroids]
    Let $\A = \{H_e, e \in E \} \subseteq \Real^r$ be a real hyperplane arrangement. For each $e \in E$, we choose a labeling $H_e^{+},H_{e}^{-}$ of the two components of $\Real^r \setminus H_e,$ and we set $H_e^0 = H_e$. For all $x \in \Real^r$, the \emph{sign vector} $\sigma(x)$ of $x$ is defined by 
    \[
        \sigma(x)(e) = \epsilon \Leftrightarrow x \in H_e^{\epsilon}, \quad \epsilon \in \{+, -, 0\}.
    \]
    The set $\L(\A) = \{\sigma(x) \mid x \in \Real^n \}$ is the set of covectors of an oriented matroid, denoted $\om(\A)$. In this situation we say that $\A$ is a \emph{representation} (or also \emph{realization}) of the oriented matroid $\om(\A).$ As was the case for unoriented matroids, oriented matroids are not necessarily representable by hyperplane arrangements.  
\end{example}

\begin{remark} Note that an oriented matroid structure is additional data on a matroid, rather than a property of it; such a structure need not exist in general. For example, the Fano matroid $\mathrm{F}_7$ from Figure~\ref{fig::fano_matroid} is a minimal example of a matroid that does not admit an orientation~\cite{BlandLasVergnas1978}. As another example, every uniform matroid admits a representation over $\Real$, which implies that uniform matroids are orientable, but in~~\cite{Ringel1956}, Ringel constructed an oriented rank $3$ uniform matroid on $9$ elements $\mathrm{Rin}(3,9)$ of $\U_{3,9}$ that gives a minimal example of an oriented uniform matroid that is not stretchable.
\end{remark}

Basic operations on matroids lift to oriented matroids as well, see~\cite[Sections~3.3 and~7.6]{bjorner1999oriented}
\begin{itemize}
    \item[(Direct Sum)] For a pair of oriented matroids $\om,\N$, their direct sum $\om\oplus \N$ is determined by
    \[
    \L(\om\oplus\N)\cong \L(\om)\times \L(\N), \quad \Topes(\om\oplus\N)\cong \Topes(\om)\times\Topes(\N)
    \]
    \item[(Restriction)]  For $S\subset E(\om)$, the restriction matroid $\om|_S$ is determined by
    \[
 \C(\om|_S)=\{C|_S\mid C\in\C,\ \supp C\subseteq S\},
 \qquad
 \La(\om|_S)=\{X|_S\mid X\in\La(\om)\}.
\]
    \item[(Contraction)] For a flat $F\subset \om$, the contraction $\om/F$ is determined by
    \[
    \La(\om/F)=
    \{X|_{E\setminus F}\mid X\in\La(\om),\ X|_F=0\}.
    \]
\end{itemize}

\subsection{Topological representation of oriented matroids}\label{sec:pseudosphere} 
Despite this negative result, oriented matroids are always representable in a generalized sense. We explain this result below. 

A \emph{signed pseudosphere} in $\Sph^{r-1}$ is a subset $S\subseteq \Sph^{r-1}$ which can be carried to an equator by a homeomorphism of $\Sph^{r-1}$, together with a choice of labels $S^+$ and $S^-$ for the two complementary components. We use the convention $\Sph^{-1}=\varnothing$; its sides in $\Sph^0$ are the two points.

A \emph{pseudosphere arrangement} in $\Sph^{r-1}$, $r\geq 1$, is a finite indexed family $\mathcal S=(S_e)_{e\in E}$ of signed pseudospheres in $\Sph^{r-1}$ such that:
\begin{itemize}
 \item[(A1)] Every intersection $S_I=\bigcap_{e\in I}S_e$ is a pseudosphere (with the convention $S_\varnothing=\Sph^{r-1}$).
 \item[(A2)] If $S_I\not\subseteq S_e$, then $S_I\cap S_e$ is a pseudosphere in $S_I$ with sides $S_I\cap S_e^+$ and $S_I\cap S_e^-$.
\end{itemize}
The arrangement is \emph{essential} if $S_E=\varnothing$, and \emph{centrally symmetric} if $-S_e=S_e$ and $-S_e^+=S_e^-$ for every $e$. Repeated pseudospheres are allowed, so parallel and antiparallel elements are included.

Put $S_e^0=S_e$. The \emph{sign vector} of $x\in\Sph^{m}$ is defined by
\[
 \sigma(x)(e)=\varepsilon
 \quad\Longleftrightarrow\quad x\in S_e^\varepsilon,
 \qquad \varepsilon\in\{+,-,0\}.
\]

\begin{theorem}[{\cite[Theorem~5.2.1]{bjorner1999oriented}}]\label{thm:topological-representation}
For an essential pseudosphere arrangement $\mathcal S$ in $\Sph^{r-1}$, the set
\[
 \La(\mathcal S)=\{\sigma(x)\mid x\in\Sph^{r-1}\}\cup\{0\}
\]
is the covector set of a loopless oriented matroid of rank $r$. Conversely, every loopless oriented matroid $\om$ admits such a representation, with $\La(\mathcal S)=\La(\om)$, and the arrangement can be chosen centrally symmetric.
\end{theorem}

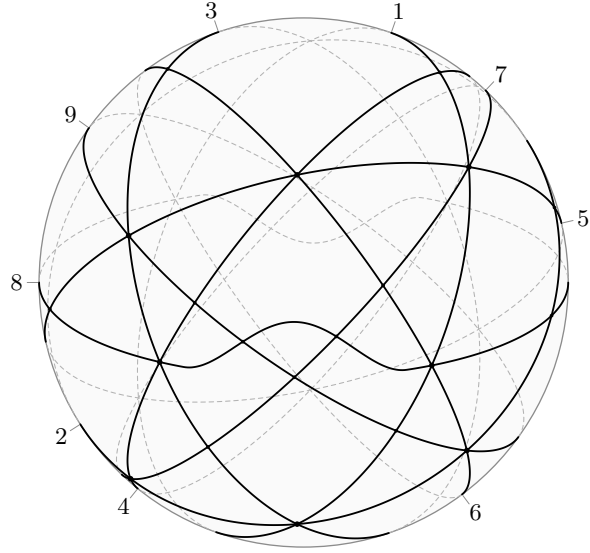
\begin{figure}[ht]
    \centering
        \begin{tikzpicture}[
  x=5cm, y=5cm,
  line cap=round, line join=round,
  npfront/.style={draw=black, line width=0.7pt},
  npback/.style={
    draw=black!30,
    line width=0.4pt,
    dash pattern=on 1.7pt off 1.4pt
  },
  every node/.style={font=\small},
  declare function={
    npbend(\t)=0.22*sin(\t)*
      (max(0,1-(cos(\t)/sin(25))^2))^2;
    npx(\t)=cos(\t+2)/sqrt(1+npbend(\t)^2);
    npy(\t)=(-sin(21)*sin(\t+2)+cos(21)*npbend(\t))
      /sqrt(1+npbend(\t)^2);
  }, scale=0.7
]
  \fill[black!2] (0,0) circle[radius=1];
  \draw[black!45, line width=0.5pt]
    (0,0) circle[radius=1];

  \foreach \sty/\lo/\hi in {
    npback/180/360, npfront/0/180
  } {
    \foreach \i/\a/\b/\c in {
      1/-5/3/3,  2/-1/3/3,  3/5/3/3,
      4/-5/-3/3, 5/-1/-3/3, 6/5/-3/3,
      7/-3/0/3,  9/2/0/3
    } {
      \pgfmathsetmacro{\nn}{sqrt(\a*\a+\b*\b+\c*\c)}
      \pgfmathsetmacro{\nx}{
        (cos(2)*\a-sin(2)*\b)/\nn}
      \pgfmathsetmacro{\ny}{
        (-sin(21)*sin(2)*\a-sin(21)*cos(2)*\b
         +cos(21)*\c)/\nn}
      \pgfmathsetmacro{\nz}{
        (cos(21)*sin(2)*\a+cos(21)*cos(2)*\b
         +sin(21)*\c)/\nn}
      \pgfmathsetmacro{\qq}{sqrt(\nx*\nx+\ny*\ny)}

      \draw[\sty]
        plot[variable=\t, domain=\lo:\hi, samples=181]
        ({(-\ny*cos(\t)-\nz*\nx*sin(\t))/\qq},
         {( \nx*cos(\t)-\nz*\ny*sin(\t))/\qq});
    }

    \pgfmathsetmacro{\start}{\lo-2}
    \pgfmathsetmacro{\finish}{\hi-2}
    \draw[\sty]
      plot[variable=\t, domain=\start:\finish, samples=361]
      ({npx(\t)},{npy(\t)});
  }

  \foreach \xx/\yy in {
    -0.024678/-0.913391,
     0.624331/ 0.436405,
     0.484256/-0.313546,
    -0.654094/-0.741836,
     0.616891/-0.635153,
    -0.544108/-0.300676,
    -0.661534/ 0.177006,
    -0.024678/ 0.406891
  } {
    \fill (\xx,\yy) circle[radius=1.5pt];
  }

  \foreach \xx/\yy in {
     0.514182/ 0.794104,
     0.349743/-0.560290,
     0.948105/ 0.283359,
     0.937069/-0.125122,
    -0.514182/ 0.806973,
    -0.361925/-0.621555,
    -0.411794/-0.076501,
    -0.959142/-0.101392,
     0.300261/-0.011261,
     0.123517/-0.203944,
    -0.034899/-0.358150,
    -0.229668/-0.223729
  } {
    \fill (\xx,\yy) circle[radius=1pt];
  }

  \foreach \i/\ang in {
    1/70.677, 2/212.414, 3/108.782,
    4/231.170, 5/12.960, 6/306.772,
    7/46.569, 8/180, 9/144.243
  } {
    \draw[black!45, line width=0.35pt]
      (\ang:1.005) -- (\ang:1.045);
    \node at (\ang:1.085) {$\i$};
  }

\end{tikzpicture}
    \caption{Pseudocircle realisation of the Non-Pappus matroid $\mathsf{N}_9$}
    \label{fig:pseudocirc_N9}
\end{figure}

See Figure \ref{fig:pseudocirc_N9} for a pseudosphere representation of the Non-Pappus matroid. If $\mathcal{S}$ is a pseudosphere arrangement, we denote by $\om(\S)$ the oriented matroid having covectors $\La(\S)$.  For an essential arrangement representing $\om$, the nonempty strata $F_X=\sigma^{-1}(X)$, $X\in\La(\om)\setminus\{0\}$, form a regular cell decomposition of $\Sph^{r-1}$, with
\[
 F_X\subseteq\overline{F_Y}\quad\Longleftrightarrow\quad X\preceq Y.
\]
In particular, the connected components of the \emph{pseudosphere complement}
\[
 \U(\mathcal S)=\Sph^{r-1}\setminus\bigcup_{e\in E}S_e
\]
are the \emph{chambers}, indexed by $\Topes(\om)$.

Coning the pseudospheres gives the central \emph{pseudohyperplane arrangement} $\A=(H_e)_{e\in E}$ in $\Real^r$, where
\[
 H_e=\{0\}\cup\{tx\mid t>0,\ x\in S_e\}.
\]
Each pair $(\Real^r,H_e)$ is homeomorphic to $(\Real^r,\Real^{r-1})$. The coned faces, together with the origin, have face poset $\La(\om)$.

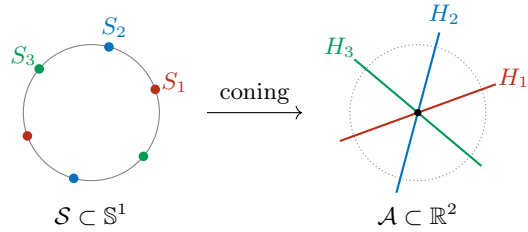
\begin{figure}[htbp]
 \centering
 \begin{tikzpicture}[scale=.9, every node/.style={font=\small}]
  \draw[gray] (0,0) circle (1);
  \foreach \a/\col/\i in {20/BrickRed/1,75/RoyalBlue/2,140/ForestGreen/3}{
   \fill[\col] (\a:1) circle (2pt);
   \fill[\col] ({\a+180}:1) circle (2pt);
   \node[text=\col] at (\a:1.3) {$S_{\i}$};
  }
  \node at (0,-1.5) {$\mathcal S\subset\Sph^1$};
  \draw[->] (1.7,0) -- node[above] {coning} (3.1,0);
  \begin{scope}[xshift=4.8cm]
   \draw[gray,densely dotted] (0,0) circle (1);
   \foreach \a/\col/\i in {20/BrickRed/1,75/RoyalBlue/2,140/ForestGreen/3}{
    \draw[\col,thick] ({\a+180}:1.22) -- (\a:1.22);
    \node[text=\col] at (\a:1.5) {$H_{\i}$};
   }
   \fill (0,0) circle (1.5pt);
   \node at (0,-1.5) {$\A\subset\Real^2$};
  \end{scope}
 \end{tikzpicture}
 \caption{A rank-two example: each pair of equally coloured antipodal points is a pseudosphere in $\Sph^1$; its cone is a line in $\Real^2$.}
 \label{fig:pseudosphere-coning}
\end{figure}

For $d\geq1$, we define the \emph{$n$-dimensional complement}, or \emph{higher complexification}, by
\begin{equation}\label{eq:pseudo-higher-complement}
 \U_n(\S) = \U_n(\A)
 :=(\Real^r)^n\setminus\bigcup_{e\in E}H_e^{\times n}.
\end{equation}
Thus $\U_1(\A)$ is the ordinary complement of the coned arrangement, homeomorphic to $(0,\infty)\times\U(\mathcal S)$ when $E\ne\varnothing$. For $n=2$, this is the microbundle complement of~\cite{deshpande2016arrangements}. For linear hyperplanes, $(x,y)\mapsto x+iy$ identifies $\U_n(\A)$ with the usual complexified complement $\mathbb C^r\setminus\bigcup_{e\in E}(H_e\otimes_{\Real}\mathbb C)$.


\subsection{Salvetti posets}\label{sec:salvetti-recollections}
A landmark result~\cite{Salvetti1987} in the theory of real arrangement complements states that the homotopy type of a complexified hyperplane arrangement complement is completely determined by its underlying oriented matroid. As we recall below, this is more generally true for higher arrangement complements of pseudosphere arrangements. The main ingredient behind this result is the so-called higher Salvetti poset of an oriented matroid, introduced in~\cite[Definition~5.27]{Tamaki2012}.


\begin{definition}\label{def:higher-salvetti}
Let $\om$ be an oriented matroid. For any $n\geq1$, the $n$-th \emph{higher Salvetti poset} is the set
\[
 \Sal_n(\om)=
 \{(X_1,\ldots,X_d)\in\La(\om)^d
   \mid X_1\leq\cdots\leq X_d,\ X_d\in\Topes(\om)\}.
\]
with order defined by
\begin{equation}\label{eq:higher-salvetti-order}
 (X_1,\ldots,X_d)\leq_{\Sal}(Y_1,\ldots,Y_d)
 \quad\Longleftrightarrow\quad
 X_i\geq X_{i-1}\circ Y_i\quad(1\leq i\leq d).
\end{equation}
We call $|\Sal_d(\om)|$ the \emph{higher Salvetti complex}.
\end{definition}
The last inequality in~\eqref{eq:higher-salvetti-order} is an equality, since both $X_d$ and $X_{d-1}\circ Y_d$ are topes. In particular, $\Sal_1(\om)$ is the discrete poset $\Topes(\om)$, and $\Sal_2(\om)$ can be identified with the classical Salvetti poset of $\om$ (see~\cite{Salvetti1987}). \\

The following theorem seems to be a folklore result. We will provide an explicit proof of that result in Section~\ref{sec:new-presheaf-models}.

\begin{theorem}[{\cite{deshpande2016arrangements}\cite{Berger2009HigherComplements}\cite{BjornerZiegler1992}}]\label{thm:higher-sal-higher-cpt}
Let $\S$ be a pseudosphere arrangement. For every $n\geq2$, the complex $|\Sal_n(\om(\S))|$ is homotopy equivalent to the higher arrangement complement $\U_n(\S)$.
\end{theorem}
\begin{proof}
    The first instance of this result is due to Salvetti~\cite{Salvetti1987}, who proved it for hyperplane arrangements and $n =2$. Deshpande~\cite{deshpande2016arrangements} states and proves this result for arbitrary pseudosphere arrangements and $n=2.$ Berger in~\cite{Berger2009HigherComplements} states this result for any $n$ and gives a sketch of proof. In~\cite[Section 8]{BjornerZiegler1992} the authors study the complexification of pseudosphere arrangement complements and prove the homotopy equivalence with $\Sal_2$. They describe in Section 9 how these results can be carried over to higher complements, without stating the exact homotopy equivalence. 
\end{proof}


\subsection{Algebraic invariants of (oriented) matroids}
Predating Salvetti's result, Orlik and Solomon~\cite{orlik1980combinatorics} have shown that for any real hyperplane arrangement $\A$ (or more generally any complex hyperplane arrangement) the cohomology ring $H^{\bullet}(\U_2(\A); \Q)$ only depends on the oriented matroid $\om(\A)$ (in fact only on the underlying matroid $\umat(\A)$) and this ring has an elegant explicit presentation. By~\cite[Theorem 5.5]{delongueville2001cohomology} this result is more generally true of $\U_n(\A)$ for any $n\geq 1.$ Let us describe these rings below. The following definition makes sense for any (possibly non-realizable) oriented matroid. 
\begin{definition}\label{def:generalized-OS}
    Let $\om$ be an oriented matroid and $n\geq 2$ an integer. We define the $n$-th \emph{generalized Orlik-Solomon algebra} $\OS^{\bullet}_n(\om)$ of $\om$ by the presentation 
    \[
        \OS^{\bullet}_n(\om) \coloneqq \frac{\Z[\omega_e \mid e \in E, \deg \omega_e = n-1]}{\langle \partial \omega_C \mid C \in \C(\om)\rangle + \langle \omega_e^2 \mid e \in E\rangle},
    \]
    where, for any signed circuit $C$ with support $\sC = \{e_1,\ldots, e_d\}$ we denote 
    \[
        \partial\omega_C \coloneqq \sum_{e \in E} C(e) \omega_{C\setminus e}, \quad \textrm{ for } n \textrm{ odd,}
    \]
    and 
    \[
        \partial\omega_C \coloneqq \sum_{e \in E} (-1)^{i} \omega_{C\setminus e_i}, \quad \textrm{ for } n \textrm{ even.}
    \]
\end{definition}
In the above definition, the polynomial rings are graded commutative. In particular, for even $n$ the algebra $\OS_n$ is anticommutative, and for odd $n$ it is commutative. In the anticommutative case, the symbol $\omega_C = \omega_{e_1}\cdots\omega_{e_d}$ only makes sense after choosing a total order on $\sC$. Without this choice, the symbol $\omega_C$ is defined up to a sign.\\

For $n=2$, the algebra $\OS^{\bullet}_2(\om)$ coincides with the classical Orlik--Solomon algebra of the matroid $\umat.$ For $n = 3,$ the algebra $\OS^{\bullet}_3(\om)$ coincides with the Cordovil algebra~\cite{Cordovil2002CommutativeAlgebraOrientedMatroids} of $\om$. As mentioned above, we have the following classical theorem. 
\begin{theorem}[{\cite[Theorem 5.2]{orlik1980combinatorics}\cite[Theorem 5.5]{delongueville2001cohomology}}]\label{thm:cohomology-arrangement-OS}
    Let $\A$ be a real hyperplane arrangement. For all $n\geqslant 2$, we have an isomorphism of graded rings 
    \[
        H^{\bullet}(\U_n(\A)) \simeq \OS^{\bullet}_n(\om(\A)).
    \]
\end{theorem}

Since generalized Orlik--Solomon algebras can be defined for arbitrary oriented matroids, it is natural to ask if the above theorem works more generally for any pseudosphere arrangements. We will show that this is true in Section~\ref{sec:deletion-contraction-cohomology-ring}, via purely combinatorial arguments. \\

\section{Weak maps of matroids and Presheaves}\label{sec:weak} In this section we introduce the categorical framework for matroid invariants. We define categories $\OM$ and $\UM$ of (un)oriented matroids and (signed) weak maps, and study presheaves and copresheaves on these categories. We list several examples of such (co)presheaves, including topes $\Topes$ and generalized Orlik--Solomon algebras $\OS_n$. We explain how degenerations of (pseudo)hyperplane arrangements give rise to weak maps and maps between their complements (defined up to homotopy), motivating the poset models constructed in Section~\ref{sec:new-presheaf-models}. 

\subsection{Weak maps of oriented matroids} 

We consider the category $\OM$ of oriented matroids whose objects are loopless oriented matroids and whose morphisms are \textit{signed weak maps}. A signed weak map $\phi\colon \om \to \N$ is a pair consisting of a ground set map $f_{\phi}\colon E(\om)\to E(\N)$ and a signed set $\epsilon_{\phi}\colon E(\om)\to \{\pm\}$ such that, for every tope $X$ of $\N$, its pullback $(\phi^*X)(e):=\epsilon_{\phi}(e)X(f_{\phi}(e))$ is a tope of $\om$, i.e., $\phi^*\Topes(\N)\subset \Topes(\om)$.

Let us list some basic examples of signed weak maps:
\begin{description}
    \item[Automorphisms] Every automorphism of an oriented matroid defines an automorphism 
    \[
        (\phi,+)\colon \om\overset{\cong}{\to} \om.
    \]
    \item[Reorientation] Reorienting a subset $S\subseteq E(\om)$ gives an isomorphism $(\id,\epsilon_{S})\colon\om\overset{\cong}{\to}\om_{-S}$, with ground-set map $\id$ and signs negative precisely on $S$.
    \item[Submatroid] A submatroid $\N\subset \om$ defines a morphism $\N\to \om$ in $\OM$. 
\end{description}

When $\phi\colon\om\to\N$ has the identity as its ground-set map and constant sign $+$, we recover the definition of a weak map of matroids in the sense of~\cite{anderson2001representing}. We provide several examples of such weak maps. Also see Section~\ref{sec:degenerations} for examples of weak maps coming from degenerations of pseudohyperplane arrangements.
\begin{example}[Canonical weak map]\label{ex:canonical-weak-map}
For any oriented matroid $\om$, we have a weak map $\U_{E(\om)} \rightarrow \om$ given by the identity and the constant reorientation $+$. Here, $\U_{E(\om)}$ denotes the free oriented matroid on $E(\om)$. This weak map will be called the \emph{canonical weak map}. 
\end{example}
\begin{example}[Operadic map]\label{ex:operadic_maps}
Let $F$ be a flat of $\om$, and let $X$ and $Y$ be topes of the contracted matroid $\om/F$ and the restricted matroid $\om|_F$ respectively. By definition, we can lift $X$ to a covector $X'$ of $\om$ with zero set $F$ and lift $Y$ to a tope $Y'$ of $\om$. The covector composition $X\oplus Y:=X'\circ Y'$ is a tope of $\om$ and does not depend on the choices of lifts, so $\Topes(\om/F)\times \Topes(\om|_F)\subset \Topes(\om)$. In other words, we have a weak map
\[
\triangle_F\colon \om\to \om/F\oplus\om|_F.
\]
\end{example}
There is a special class of weak maps called \textit{framings}; see~\cite[Sec~5.4.1]{Anderson_2025} for details. We are interested in the following example of framing, which will play an important role in Section~\ref{sec:deletion-contraction-cohomology-ring}.
\begin{example}[Deletion-Contraction]\label{ex:deletion_contraction_OM} For an oriented matroid $\om$ and an element $e\in E({\om})$ that is a rank-one flat, every tope of the contraction $\om/e$ lifts to a covector of $\om$ and hence to a tope of $\om$ that restricts to a tope of the deletion $\om\setminus e$. In other words, we have a weak map $\om\setminus e\to \om/e$.
\begin{lemma}
The weak map $\om\setminus e\to \om/e$ fits into a pullback square in $\OM$
\begin{equation}
    \begin{tikzcd}
    	{\om\setminus e} & \om \\
    	{\om/e} & {\om/e\oplus\om|_e}
    	\arrow[from=1-1, to=1-2]
    	\arrow[from=1-1, to=2-1]
    	\arrow["\lrcorner"{anchor=center, pos=0.125}, draw=none, from=1-1, to=2-2]
    	\arrow[from=1-2, to=2-2]
    	\arrow[from=2-1, to=2-2]
        \end{tikzcd}
\end{equation}
\end{lemma}
\begin{proof} For a tope $X$ of $\om\setminus e$, its preimage with respect to the restriction map $\res\colon \Topes(\om)\to \Topes(\om\setminus e)$ consists of two elements if $X$ is a tope of $\om/e$ and one element otherwise. This means that we have a pushout square at the level of tope sets
\begin{equation}\label{diag:topes_pushout}
    \begin{tikzcd}
	{\Topes(\om/e)\times \Topes(\om|_e)} & {\Topes(\om)} \\
	{\Topes(\om/ e)} & {\Topes(\om\setminus e)}
	\arrow[from=1-1, to=1-2]
	\arrow[from=1-1, to=2-1]
	\arrow[from=1-2, to=2-2]
	\arrow[from=2-1, to=2-2]
    \arrow["\lrcorner"{anchor=center, pos=0.125, rotate=180}, draw=none, from=2-2, to=1-1]
\end{tikzcd}
\end{equation}  Let $\phi\colon \N\to \om/e$ and $\psi\colon \N\to \om/e$ be signed weak maps that make the deigram commute. At the level of ground sets and signs, $\phi$ and $\psi$ are given by the same data $(f\colon E(\N)\to E(\om)\setminus\{e\},\epsilon\colon E(\N)\to \{\pm\})$, and, by the universal property of~\eqref{diag:topes_pushout}, these data determine a unique signed weak map $\N\to \om\setminus e$.
\end{proof} 
\end{example}

\subsection{Presheaves on oriented matroids}
Since, for a pair of oriented matroids, we have $\Topes(\om\oplus \N)\cong \Topes(\om)\times\Topes(\N)$, direct sum is the categorical coproduct in $\mathrm{OM}$. The initial object in $\OM$ is the empty matroid $\varnothing=\U_{0,0}$\footnote{We use convention $\Topes(\varnothing)=\{*\}$}. 

Let $\C=(\C,\otimes,1_{\C})$. In most cases, by $\C$ we mean the category of sets/topological spaces/posets/categories or the category of vector
spaces under the tensor product $(\mathsf{Vect}_{\Field},\otimes,\Field)$. 
\begin{definition}
An $\OM$-presheaf with values in $\C$ is a contravariant monoidal functor
\[
\O\colon (\OM,\oplus,\varnothing)^{\op}\to (\C,\otimes, 1)
\] In a dual way, we define an $\OM$-copresheaf as a monoidal covariant functor. We denote by $\PSh_{\OM}(\C)$ the category of $\C$-valued $\OM$-presheaves.
\end{definition} 

In other words, given an $\OM$-presheaf $\O$, for every signed weak map $\phi\colon \om\to \N$, we have a morphisms $\phi^*\colon \O(\N)\to \O(\om)$, and the presheaf satisfies a decomposition $\O(\om\oplus\N)\cong \O(\om)\otimes\O(\N)$.  

\begin{example}[Topes]
By construction, topes of oriented matroids form an $\OM$-presheaf
\[
\Topes(-)\colon \OM^{\op}\to \mathrm{Sets}
\]Moreover, this presheaf is represented by an oriented matroid $\U_{1,1}$
\[
\Topes(-)\cong \Hom_{\OM}(-;\U_{1,1})
\]
By the Yoneda lemma, this $\OM$-presheaf is determined by the universal property
\[
\Hom_{\mathsf{PSh}_{\OM}}(\Topes,\O)\cong \O(\U_{1,1})
\]
\end{example}
The following example will be central in this article. 
\begin{prop}\label{prop:presheaf-Orlik-Solomon}
For every $n\geq 2$, the generalized Orlik--Solomon algebras $\OS_n$ form a $\OM$-copresheaf
\[
\OS_n\colon\OM\to \mathsf{grCAlg}
\] defined by the rule
\[
\OS_n(\phi)\colon \OS_n(\om)\to \OS_n(\N), \quad \omega_e\mapsto \epsilon^{n}_{\phi}(e)\omega_{f_{\phi}(e)}.
\]
\end{prop}
\begin{proof} The proposed assignment preserves the square-zero relations, so it remains to check the circuit relations. Let $C$ be a signed circuit of $\om$, with ordered support $\{e_1,\ldots,e_d\}$, and put
\[
 a_i=f_{\phi}(e_i),\qquad s_i=\epsilon_{\phi}(e_i)C(e_i)\in\{\pm1\}.
\]
The image of $\partial\omega_C$ is $R$ when $n$ is even, and $(\prod_{i=1}^d\epsilon(e_i))R$ when $n$ is odd, where
\[
 R=
 \begin{cases}
 \displaystyle\sum_{i=1}^d(-1)^i\prod_{j\ne i}\omega_{a_j},&n\text{ even},\\[6pt]
 \displaystyle\sum_{i=1}^d s_i\prod_{j\ne i}\omega_{a_j},&n\text{ odd}.
 \end{cases}
\]
All products are taken in increasing order of their indices, retaining repeated occurrences. We show that $R=0$ in $\OS_n(\N)$.

First, for any signed circuit $D$ of $\N$, its monomial $\omega_D$ vanishes: multiplying $\partial\omega_D=0$ by $\omega_a$, for any $a\in\underline D$, leaves a single term $\pm\omega_D$, since every other term contains $\omega_a^2$.

If $n$ is even and $a_i=a_j$ for distinct indices, reorder the support of $C$ so that these two occurrences come first; this changes $R$ only by an overall sign. The two summands omitting these occurrences cancel, and every other summand contains $\omega_{a_i}^2$. Hence $R=0$ whenever an image label is repeated. If $n$ is odd and $a_i=a_j$ with $s_i=-s_j$, the same cancellation holds because the generators commute. Thus we may assume that the signs $s_i$ define a sign vector $X$ on $A=\{a_1,\ldots,a_d\}$ by $X(a_i)=s_i$; for even $n$, the labels $a_i$ are moreover distinct.

The sign vector $X$ cannot be a tope of $\N|_A$. Indeed, such a tope extends to a tope $T$ of $\N$: lift it to a covector of $\N$ and compose with any tope. Then $\phi^*T$ would be a tope of $\om$ agreeing with $C$ on $\underline C$, contradicting circuit--covector orthogonality. Since $X$ has full support on $A$ but is not a covector of $\N|_A$, some signed circuit $D$ of this restriction is not orthogonal to $X$. Reversing $D$ if necessary gives $D\preceq X$. This is also a signed circuit of $\N$.

Choose a set of indices $I\subseteq\{1,\ldots,d\}$ containing exactly one occurrence of each element of $\underline D$, and put $J=\{1,\ldots,d\}\setminus I$. For even $n$, reorder once more so that the indices in $I$ come first, changing $R$ only by an overall sign. In either parity, the summands of $R$ indexed by $I$ add up to
\[
 (\partial\omega_D)\prod_{j\in J}\omega_{a_j}=0.
\]
Here the order on $\underline D$ is induced by $I$, and for odd $n$ the coefficients agree because $D(a_i)=s_i$ for $i\in I$. Every remaining summand contains all the selected occurrences indexed by $I$, so is a multiple of $\omega_D=0$. This also handles repeated image labels with equal signs when $n$ is odd. Thus every circuit relation maps to zero, and $\OS_n(\phi)$ is well defined. The identity and composition laws follow on generators.

Finally, the circuits of a direct sum are precisely the circuits of its two summands, so the presentations give natural isomorphisms
\[
 \OS_n(\om\oplus\N)\cong\OS_n(\om)\otimes\OS_n(\N),
 \qquad \OS_n(\varnothing)\cong\Z,
\]
compatible with the maps above. Hence $\OS_n$ is a monoidal covariant functor.
\end{proof}

\subsection{Degenerations of pseudohyperplane arrangements}\label{sec:degenerations}Geometrically, weak maps correspond to degenerations of pseudohyperplane arrangements. We give an explicit definition of this notion below. 
\begin{definition} A \textit{degeneration} of pseudohyperplane arrangements, denoted $\A\rightsquigarrow \A'$, is a family of pseudohyperplane arrangements $\A(t) = \{(H_e(t), H^{+}_e(t), H^{-}_e(t)), e \in E\}$ in $\Real^r$ indexed by $t\in [0,1]$, with $\A'=\A(0)$ and $\A=\A(1)$, such that 
    \begin{enumerate} 
        \item[(Continuity)] For every $e\in E$, there exists a continuous function 
        \[
            F_e:[0,1]\times \Real^r \rightarrow \Real,
        \]
        such that $H_e(t) = \{F(t, \cdot) = 0\}$ and $H^{\pm}_e(t) = \{\pm F_e(t, \cdot) > 0 \}$.
        \item[(Trivialization)] There exists a homeomorphism 
        \[
            \Phi:(0,1]\times \Real^r \xrightarrow{\simeq}(0,1]\times \Real^r
        \]
        of the form $\Phi(t,x) = (t, h_t(x))$ with $h_1 = \Id$, such that 
        \[
            h_t(H^{\epsilon}_e(1)) = H^{\epsilon}_e(t), \quad \forall e \in E, \forall t\in (0,1], \forall\epsilon \in \{0, +, -\}.
        \]
    \end{enumerate}
\end{definition}
\begin{lemma}
A degeneration $\A\rightsquigarrow\A'$ induces a weak map of oriented matroids 
\[
    \om(\A) \rightarrow \om(\A').
\]
\end{lemma}
\begin{proof}
Let $T$ be a tope of $\om(\A')=\om(\A(0))$ and $x_T \in \Real^r$ an element of the corresponding region of $\U(\A')$, i.e.,
\[
    \forall e\in E: T(e)F_e(0,x_T) > 0.
\]
By continuity of all the $F_e$'s, there exists $\epsilon > 0$ such that 
\[
 \forall e\in E:  T(e)F_e(\epsilon,x_T) > 0,
\]
for all $e\in E$. Using the trivialization $\Phi$, we obtain the existence of some element $x_T' \in \Real^r$ such that 
\[
\forall e\in E:    T(e)F_e(1,x'_T) >0,
\] 
so $T$ is a tope of $\om(\A)=\om(\A(1))$ as well.
\end{proof}
Let us consider a typical example of a degeneration.
\begin{example}[Pappus pushout square] Let us consider the Non-Pappus pseudocircle configuration from Figure~\ref{fig:pseudocirc_N9}. The degeneration from Figure~\ref{fig:deformation_pappus} produces a weak morphism of oriented matroids $\mathsf{N}_9\to\mathsf{P}_9$.
\begin{figure}
\centering
\resizebox{0.98\linewidth}{!}{%
\begingroup
\def\NPamplitude{0.22}
\begin{tikzpicture}[
  x=2.75cm, y=2.75cm,
  line cap=round, line join=round,
  every node/.style={font=\small},
  npfront/.style={draw=black, line width=0.48pt},
  npback/.style={draw=black!28, line width=0.3pt,
    dash pattern=on 1.3pt off 1.2pt},
  npmotion/.style={-{Stealth[length=1.3mm,width=0.95mm]},
    line width=0.45pt, shorten <=1.3pt},
  declare function={
    npbend(\t,\e)=\e*sin(\t)*
      (max(0,1-(cos(\t)/sin(25))^2))^2;
    npx(\t,\e)=cos(\t+2)/sqrt(1+npbend(\t,\e)^2);
    npy(\t,\e)=(-sin(21)*sin(\t+2)+cos(21)*npbend(\t,\e))
      /sqrt(1+npbend(\t,\e)^2);
  }
]
  \foreach \offset/\deform/\letter in {
    0/\NPamplitude/N, 3.24/0/P
  } {
    \begin{scope}[shift={(\offset,0)}]
      \fill[black!2] (0,0) circle[radius=1];
      \draw[black!45,line width=0.4pt]
        (0,0) circle[radius=1];

      \foreach \sty/\lo/\hi/\weight/\shade in {
        npback/180/360/0.55/35, npfront/0/180/0.95/100
      } {
        \foreach \i/\a/\b/\c in {
          1/-5/3/3,  2/-1/3/3,  3/5/3/3,
          4/-5/-3/3, 5/-1/-3/3, 6/5/-3/3,
          7/-3/0/3,  9/2/0/3
        } {
          \pgfmathsetmacro{\nn}{sqrt(\a*\a+\b*\b+\c*\c)}
          \pgfmathsetmacro{\nx}{
            (cos(2)*\a-sin(2)*\b)/\nn}
          \pgfmathsetmacro{\ny}{
            (-sin(21)*sin(2)*\a-sin(21)*cos(2)*\b
             +cos(21)*\c)/\nn}
          \pgfmathsetmacro{\nz}{
            (cos(21)*sin(2)*\a+cos(21)*cos(2)*\b
             +sin(21)*\c)/\nn}
          \pgfmathsetmacro{\qq}{sqrt(\nx*\nx+\ny*\ny)}
          \draw[\sty,name path global=np-\i]
            plot[variable=\t,domain=\lo:\hi,samples=181]
            ({(-\ny*cos(\t)-\nz*\nx*sin(\t))/\qq},
             {( \nx*cos(\t)-\nz*\ny*sin(\t))/\qq});
        }

        \ifdim\deform pt>0pt
          \pgfmathsetmacro{\gstart}{\lo+65}
          \pgfmathsetmacro{\gfinish}{\lo+115}
          \draw[draw=black!\shade,line width=0.3pt,
                densely dotted,opacity=0.45]
            plot[variable=\t,domain=\gstart:\gfinish,samples=101]
            ({npx(\t,0)},{npy(\t,0)});
        \fi

        \pgfmathsetmacro{\start}{\lo-2}
        \pgfmathsetmacro{\finish}{\hi-2}
        \draw[\sty,draw=black!\shade,line width=\weight pt,
              name path global=np-8]
          plot[variable=\t,domain=\start:\finish,samples=361]
          ({npx(\t,\deform)},{npy(\t,\deform)});

        \ifdim\deform pt>0pt
          \pgfmathsetmacro{\mt}{\lo+90}
          \pgfmathsetmacro{\mend}{0.20*\deform}
          \draw[npmotion,draw=black!\shade]
            plot[variable=\e,domain=\deform:\mend,samples=17]
            ({npx(\mt,\e)},{npy(\mt,\e)});
        \fi
      }

      \foreach \xx/\yy in {
        -0.024678/-0.913391,  0.624331/ 0.436405,
         0.484256/-0.313546, -0.654094/-0.741836,
         0.616891/-0.635153, -0.544108/-0.300676,
        -0.661534/ 0.177006, -0.024678/ 0.406891
      } {\fill (\xx,\yy) circle[radius=1.15pt];}

      \foreach \xx/\yy in {
         0.514182/ 0.794104,  0.349743/-0.560290,
         0.948105/ 0.283359,  0.937069/-0.125122,
        -0.514182/ 0.806973, -0.361925/-0.621555,
        -0.411794/-0.076501, -0.959142/-0.101392,
         0.300261/-0.011261
      } {\fill (\xx,\yy) circle[radius=0.8pt];}

      \ifdim\deform pt>0pt
        \path[name intersections={of=np-7 and np-8,by=np78}];
        \path[name intersections={of=np-9 and np-8,by=np89}];
        \fill (np78) circle[radius=0.8pt];
        \fill (np89) circle[radius=0.8pt];
        \fill (-0.034899,-0.358150) circle[radius=0.8pt];
        \node[font=\scriptsize,anchor=west,
              fill=black!2,inner sep=0.6pt]
          at (-0.012,-0.255) {$\varepsilon$};
      \else
        \fill (-0.034899,-0.358150) circle[radius=1.5pt];
      \fi

      \foreach \i/\ang in {
        1/70.677, 2/212.414, 3/108.782,
        4/231.170, 5/12.960, 6/306.772,
        7/46.569, 8/180, 9/144.243
      } {
        \draw[black!45,line width=0.3pt]
          (\ang:1.005) -- (\ang:1.045);
        \node at (\ang:1.09) {$\i$};
      }
      \ifdim\deform pt>0pt
        \node[font=\normalsize] at (0,-1.21)
          {$\mathsf{N}_9\quad(\varepsilon>0)$};
      \else
        \node[font=\normalsize] at (0,-1.21)
          {$\mathsf{P}_9\quad(\varepsilon=0)$};
      \fi
    \end{scope}
  }

  \draw[-{Stealth[length=2mm,width=1.4mm]},line width=0.65pt]
    (1.19,0.12) -- node[above=3pt,font=\normalsize]
    {$\varepsilon\to0$} (2.02,0.12);
\end{tikzpicture}
\endgroup

 }
   
    \caption{Pappus circle configuration as a degeneration of Non-Pappus pseudocircle configuration}
    \label{fig:deformation_pappus}
\end{figure}
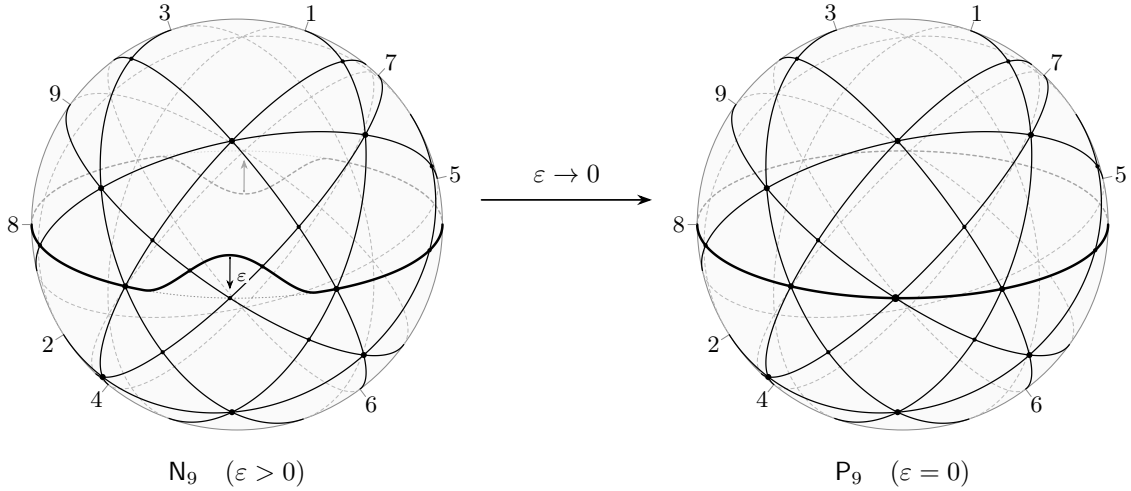
It is obtained from a local deformation that transforms an independent set $\{7,8,9\}$ into a dependent set of rank $2$. This observation tells us that this morphism fits into a pushout square
\[\begin{tikzcd}
	{\U_{3,3}} & {\U_{2,3}} \\
	{\mathsf{N}_9} & {\mathsf{P}_9}
	\arrow[from=1-1, to=1-2]
	\arrow[from=1-1, to=2-1]
	\arrow[from=1-2, to=2-2]
	\arrow[from=2-1, to=2-2]
	\arrow["\lrcorner"{anchor=center, pos=0.125, rotate=180}, draw=none, from=2-2, to=1-1]
\end{tikzcd}\]
\end{example}

However, according to our literature search, it is not clear when weak maps that are the identity on the ground set, are induced by sequences of degenerations. We refer to~\cite{dobbins2025homotopyequivalencegrassmanniansmacphersonians} for a study of the case of rank $3$ matroids.

Our interest in the notion of degeneration stems from the following considerations. For any degeneration $\A(t)$ and integer $n\geq 1$, one can consider the \emph{total higher complement} $\W_n$ of $\A$, defined as 
\[
    \W_n = \{ (t,x) \in [0,1]\times (\Real^r)^n \mid x \in \U_n(\A(t))\}.
\] Note that $\W_n$ is a topological manifold with boundary $\partial\W_n=\W_n|_{t=0}\cup \W_n|_{t=1}\cong \U_n(\A')\cup \U_n(\A)$, and interior $\W^{\circ}_n=\W_n\cap \{t\neq 0,1\}$
\begin{lemma}\label{lem:geometric_lemma}
The inclusion
\[
\W^{>0}_n:=\W_n\cap\{t>0\}\hookrightarrow \W_n 
\] is a homotopy equivalence.
\end{lemma}
\begin{proof}
By~\cite[Theorem~2, p.~339]{brown1962collars}, the boundary $\partial\W_n$ is collared in $\W_n$, which implies that the inclusion $\W_n^{\circ} \hookrightarrow \W_n$ is a homotopy equivalence (see the argument in the proof of ~\cite[Theorem 9.26]{lee2013smooth}). By the same argument, the inclusion $\W_n^{\circ} \hookrightarrow \W_n^{>0}$ is also a homotopy equivalence, which implies the desired equivalence by the two-out-of-three property.   
\end{proof}
Thanks to the trivialization of $\A(t)$ over $(0,1]$, the inclusion $\U_n(\A)\subset \W^{>0}_n$ is a deformation retract, so, by Lemma~\ref{lem:geometric_lemma}, the inclusion $\U_n(\A)\subset \W_n$ is a homotopy equivalence as well. Hence, we can invert this inclusion in the homotopy category of spaces to obtain a map
\begin{equation}\label{eq::homotopy_map_complements}
        \U_n(\A') \overset{\iota_0}{\hookrightarrow} \W_n \overset{\iota^{-1}_1}{\rightarrow} \U_n(\A).
\end{equation}
We show in Section~\ref{sec:weight_topes} that this map is completely determined by the induced weak map 
\[
\om(\A) \rightarrow \om(\A'),
\]
and construct an $\OM$-presheaf $\Topes_n$ that models these maps.

\begin{remark} Let us note that the existing model given by the higher Salvetti posets $\Sal_n(\om)$ described in~\ref{sec:salvetti-recollections} is ill-suited for this task. The reason is that, for a signed weak map $\phi$, the naive correspondence 
\[
(X_1\leq X_2\leq \cdots \leq X_n)\mapsto(\phi^*X_1\leq \cdots \phi^*X_n) 
\] forms a well-defined map $\phi^*\colon \Sal_n(\N)\to \Sal(\om)$ only if $\phi$ is a \textit{strong map}, i.e., if the pullback preserves covectors $\phi^*\La(\N)\subset \La(\om)$. Note that not all weak maps/degenerations are strong. For example, there is a degeneration of real line arrangements that produces a weak map $\triangle_3\colon \U_{2,3}\to \U_{1,2}\oplus \U_{1,1}$
\[
\{\{x\},\{y\},\{tx+y\}\} \overset{t\to 0}{\rightsquigarrow}\{\{x\},\{y\},\{y\}\}.
\] We see that $(0,0,+)$ is a covector of $\U_{1,2}\oplus \U_{1,1}$, but is not a covector of $\U_{2,3}$.
\end{remark}

\subsection{Presheaves on unoriented matroids} We can develop a theory of presheaves for unoriented matroids as follows. 

Let $\UM$ be the category whose objects are loopless unoriented matroids and whose morphisms are\textit{weak maps}: a map $f\colon E(\mat)\to E(\Nmat)$ is weak if  $\rk_{\Nmat} f(S)\leq \rk_{\mat} S$ for every $S\subseteq E(\mat)$; see~\cite[Sec~2.1]{stamps2013topological}.

All examples of morphisms discussed in the oriented case exist on the unoriented level as well. As an example, for a flat $F$ of $\mat$, we have a weak map $\triangle_F\colon \mat\to \mat/F\oplus\mat|_F$ that is an identity on the level of ground sets, since
\[
    \rk_{\mat/F\oplus \mat|_F}(S)=\rk_{\mat/F}(S\setminus F)+\rk _{\mat|_F}(S\cap F)=\rk_{\mat}(S\cup F)+\rk_{\mat}(S\cap F)-\rk_{\mat}(F)\leq \rk_{\mat}(S)
\] 
\begin{claim*}
Assigning to an oriented matroid its underlying matroid $\umat$ defines a functor
\[
\underline{(-)}\colon\OM\to \UM, \quad \om\mapsto \underline{\om}, \quad \phi=(f_{\phi},\epsilon_{\phi})\mapsto f
\]    
\end{claim*}
\begin{proof}
It suffices to show that, for $S\subset E(\om)$, if $f_{\phi}|_S$ is injective and $f_{\phi}(S)$ is independent, then $S$ is independent as well. Since $f_{\phi}(S)$ is independent, every sign pattern on $f(S)$ extends to a tope of $\N$. Pulling these topes along $\phi$ shows that $\Topes(\om|_S)=\{\pm\}^S$, so $\om|_S$ is free, i.e., $S$ is independent.
\end{proof}

Similarly to the oriented case, we define a $\C$-valued $\UM$-presheaf as a contravariant monoidal functor
\[
\O\colon (\UM,\oplus,\varnothing)^{\op}\to (\C,\otimes, 1)
\] Thanks to the underlying matroid functor, every $\UM$-presheaf $\O$ defines an $\OM$-presheaf $\underline{\O}$ by the rule $\underline{\O}(\om):=\O(\umat)$.
\begin{example}[$\OS$-algebras]
Recall that the algebra $\OS$ is defined for every unoriented matroid. Moreover, as in the oriented case, $\OS$ algebras form an $\UM$-copresheaf $\OS\colon \UM\to \mathsf{grCAlg}$ by the rule
\[
\OS(f)\colon \OS(\mat)\to \OS(\Nmat), \quad \omega_e\mapsto \omega_{f(e)}.
\] Indeed, we have to check that, for a circuit $C$ of $\om$, the image $\OS(f)(\partial\omega_{C})$ is zero in $\OS(\Nmat)$. If $|f(C)|<|C|$, then the monomial $\OS(f)(\omega_{C})$ contains a square monomial $\omega^2_e$, so $\OS(f)(\partial\omega_{C})$. If $|f(\C)|=|\C|$, then $\rk_{\Nmat}f(\C)\leq \rk_{\mat}\C<|\C|=|f(\C)|$, which implies that $f(\C)$ is dependent set and hence $f(\partial\omega_{\C})=\partial\omega_{f(\C)}=0$. 

In particular, this implies that the $\OM$-copresheaf $\OS_n$, for even $n$, is obtained from the $\UM$-presheaf by precomposing with the underlying matroid functor.
\end{example}

We finish this subsection by considering three additional examples.
\begin{example}[Rainbow circuit free $r$-colorings]
For a matroid $\om$, a rainbow circuit free $r$-colorings of a matroid $\mat$ is a coloring of the ground set $c\colon E(\mat)\to [r]$ such that every circuit contains two elements of the same color~\cite{berczi2020rainbow}. Equivalently, every subset on which $c$ is injective is independent, or $|c(S)|\leq \rk_{\mat} S$ for all $S\subset E(\om)$. In other words, it determines a weak map to the rank $r$ free matroid $c\colon\mat\to \U_{r,r}$. Such colorings therefore form the representable $\UM$-presheaf $\Hom_{\UM}(-;\U_{r,r})$.
\end{example}
\begin{example}[Circuit varieties]
For a vector space $\Field^d$ and a matroid $\mat$, its (affine) \textit{circuit variety} is an affine algebraic variety~\cite{ClarkeGraceMohammadiMotwani2023}
\[
\C^{\mathrm{aff}}_d(\mat)=\{(v_e)\in \Field^d\mid\dim \langle v_e\rangle_{e\in S}\leq \rk_{\mat}S \text{ for every }S\subset E(\mat)\}.
\] If we replace vectors with points in projective space $\Pro_{\Field}^{d-1}$, we obtain a projective version $\C^{\mathrm{proj}}_{d-1}(\mat)$. Note that, at the level of sets, we have
\[
\C^{\mathrm{aff}}_d(\mat)=\Hom_{\UM}(\mat;\mathrm{AG}(d;\Field)),\quad \C^{\mathrm{proj}}_d(\mat)=\Hom_{\UM}(\mat;\mathrm{PG}(d;\Field)),
\] where $\mathrm{AG}(d;\Field)$ and $\mathrm{PG}(d;\Field)$ are affine/projective geometry matroids. Moreover, we can see that the induced $\UM$-presheaf structures indeed are morphisms of algebraic varieties. So, $\C^{\mathrm{aff}}_d$ and $\C^{\mathrm{proj}}_d$ form $\UM$-presheaves in the category of algebraic varieties.
\end{example}

\begin{example}[Holonomy Lie algebras]\label{ex:holonomy-presheaf}
For a matroid $\mat$, its integral holonomy Lie algebra $\h(\mat)$ is the free Lie algebra on $\OS^1(\mat)^*$ modulo the ideal generated by the image of the dual multiplication map
\[
 \OS^2(\mat)^*\longrightarrow
 \bigwedge^2\OS^1(\mat)^*,
\] see~\cite{suciu2025resonance} for details. Taking duals of the $\UM$-copresheaf structure on $\OS$ induces an $\UM$-presheaf structure on $\h$ in the category of Lie algebras under direct sums 
\[
\h\colon (\UM,\oplus,\varnothing)^{\op}\to (\mathsf{LieAlg},\oplus,0)
\] Explicitly, for a simple matroid $\mat$, $\h(\mat)$ has a presentation 
\[
 \h(\mat)=\Lie\langle x_e:e\in E(\mat)\rangle\big/
 \left\langle\left[x_e,\sum_{u\in F}x_u\right]:
 e\in F,\ F\in\Fl^{\rk F=2}(\mat),\ \rk F=2\right\rangle.
\] And for a weak map $f\colon \mat\to \Nmat$ of simple matroids, we have
\[
f^*\colon \h(\Nmat)\to \h(\mat), \quad x_e\mapsto \sum_{u\in f^{-1}(e)} x_u
\] For a nonsimple matroid, use one generator for each rank-one flat; equivalently $\h(\mat):=\h(\si\mat)$.
\end{example}

\subsection{Remark on simplifications and localisations} Note that all the examples of $\OM/\UM$-presheaves considered are unchanged with respect to simplifications. In this section we give the precise localisations of matroid categories through which the considered examples factor.

For example, in the unoriented case, for any matroid we have a simplification weak map $\si\colon \mat\to \si\mat$. Note that simplifications are functorial and determine a functor
\[
\si\colon\UM\to \UM_{\mathrm{simp}},
\] where $\UM_{\mathrm{simp}}$ is the full subcategory of unoriented matroids. Moreover, $\si$ is a right adjoint functor to the embedding $\UM_{\mathrm{simp}}\subset \UM$, $\Hom_{\UM_{\mathrm{simp}}}(\si\mat,\Nmat)\cong \Hom_{\UM}(\mat,\iota(\Nmat))$. All in all, we have
\begin{prop}
For an $\UM$-presheaf $\O$, if $\si^*\colon \O(\mat)\to \O(\si\mat)$ is an isomorphism for every $\mat$, then $\O$ factors through simple matroids
\[
\O\colon \UM\overset{\si}{\to} \UM_{\mathrm{simp}}\to \C.
\]
\end{prop}

In contrast to the unoriented case, simplifications of oriented matroids are not canonical, because of antiparallel elements. Instead of this, we consider a localisation. Let $\W_{\Topes}$ be the class of signed weak maps whose pullbacks $\phi^*\colon \Topes(\N)\overset{\cong}{\to} \Topes(\om)$ are isomorphisms. Note that all examples of $\OM$-presheaves described above factor through the corresponding localisation $\OM[\W_{\Topes}^{-1}]$. The following proposition characterizes this category. The proof is straightforward but lengthy, and we omit it.

\begin{prop} The following holds in $\OM[\W_{\Topes}^{-1}]$
\begin{itemize}
    \item[(i)] We have an explicit description of morphisms
    \[
    \Hom_{\OM[\W_{\Topes}^{-1}]}(\om,\N)=\Hom_{\OM}(\om,\N)/\sim,\] 
    where $\phi\sim \psi$ if  $\phi^*=\psi^*\colon \Topes(\N)\to \Topes(\om)$
    \item[(i')] Further, if $\N$ is simple, we have 
    \[
    \Hom_{\OM[\W_{\Topes}^{-1}]}(\om,\N)=\Hom_{\OM}(\om,\N)
    \]
    \item[(ii)] Every oriented matroid is isomorphic to a simple one
    \item[(iii)] All in all, the full subcategory $\OM_{\mathrm{simp}}\subset \OM$ of simple oriented matroids embeds fully faithfully in the localisation, giving an equivalence
    \[
    \OM_{\mathrm{simp}}\cong \OM[\W_{\Topes}^{-1}]
    \] 
\end{itemize} In particular, if an $\OM$-presheaf $\O$ sends weak equivalences to isomorphisms, then it admits a factorisation
\[
\O\colon \OM\to \OM[\W^{-1}_{\Topes}]\simeq \OM_{\mathrm{simp}}\to \C
\]
\end{prop}

\section{Poset models for higher arrangement complements}\label{sec:new-presheaf-models}
In this section we construct new poset models $\Topes_n$ and $\WTopes_n$ for higher arrangement complements and compare them with the Salvetti and Delucchi-Falk models. We show that these new models carry an $\OM$-presheaf structure modeling the maps between arrangement complements induced by degenerations (see Section~\ref{sec:degenerations}).

\subsection{Homotopy theory of posets}\label{sec:sub:Poset} 
Let us recall a few results and notions in the homotopy theory of posets; we refer to~\cite{barmak2011algebraic} for details.

For a finite poset $\Po$, write $|\Po|$ for the geometric realization of its nerve $\mathsf{N}(\Po)$, equivalently of its order complex $\Delta(\Po)$. We say that a poset map $\phi\colon \Po\to \Qo$ is a (homotopy) equivalence if $|\phi|\colon |\Po|\to|\Q|$ is a homotopy equivalence.

Instead of geometric realisations, we can equip a poset $\Po$ with the lower Alexandrov topology, whose basis consists of the lower ideals $\{\downarrow p\}_{p\in \Po}$. These two models are equivalent, since there is a functorial weak equivalence of topological spaces $|\Po|\overset{\simeq}{\to} \Po$~\cite[Thm~2]{McCord1966}.

The following theorems will be our main tools to prove that two posets are homotopy equivalent. The first is a poset version of Quillen's Theorem of type A
\begin{theorem}[{\cite[Proposition~1.6]{Quillen1978}}]\label{thm:quillen-type-A}
A map of posets $\phi\colon \Po\to \Qo$ is an equivalence if for all $q \in \Qo$ the lower fiber 
\[\phi^{-1}(\downarrow q):=\{p\in \Po \mid \phi(p)\leq q\}\]
is contractible. 

More generally, for a continuous map $f\colon X\to \Po$ from a topological space to a finite poset, if, for every $p\in \Po$, the preimage $\phi^{-1}(\downarrow p)$ is contractible, then $\phi$ is a weak equivalence.
\end{theorem} 
The second is a special case of the Alexandrov nerve theorem
\begin{theorem}[{\cite{alexandrov1926kombinatorische}}]\label{thm:cover-contractible}
    If a finite poset $\Po$ admits a covering $\Po=\cup_{e\in E} \mathcal{I}_{e} $ by lower ideals such that, for all $\varnothing\neq S\subseteq E$, the intersection $\cap_{i\in S} \mathcal{I}_{e}$ is nonempty and contractible, then $\Po$ is contractible.
\end{theorem}  

\subsection{Weighted Topes}\label{sec:weight_topes} In this section, we construct an $\OM$-presheaf of weighted topes $\Topes_n$ in posets that forms a model for higher complement arrangements $\U_n(\A)$ similarly to Salvetti posets, and show that $\Topes_n$ carry the desired presheaf structure alluded to in Section~\ref{sec:degenerations}. 

\begin{defi}
Let $\om$ be an oriented matroid and $n\geq 1$ an integer. The $n$-th \emph{poset of weighted topes} of $\om$, denoted $\Topes_n(\Mat)$, is the poset with underlying set $\Topes(\Mat)\times [n]^{E(\om)}$ and relations 
\begin{equation}\label{eq:def-order-weighted-tope}
(X_1,w_1)\leq (X_2,w_2) \quad \stackrel{\mathrm{def}}{\Leftrightarrow} \quad  
	\forall e\in E(\om)\colon (w_1(e)=w_2(e), X_1(e)=X_2(e))\text{ or } (w_1(e)<w_2(e)).
\end{equation}
\end{defi} 
Similarly to the usual topes, weighted topes form an $\OM$-presheaf in posets
\begin{equation}\label{prop:presheaf-structure-weighted-topes}
\Topes_n\colon \OM^{\op}\to \Pos,
\end{equation} defined by the rule
\[
\phi^*\colon \Topes_n(\N)\to \Topes_n(\om), \quad (X,w_X)\mapsto (\phi^*X,f^*_{\phi}w_X),
\]
Let us explain the relation of weighted topes to higher complements arrangements. Recall from Section~\ref{sec:pseudosphere} that for any pseudohyperplane arrangement $\A = \{H_e, e \in E\}$ in $\Real^r$ and any $x \in \Real^r$, we have defined the sign vector $\sigma(x) \in \{+,-,0\}$ which records on which side of $\Real^r \setminus H_e$ the point $x$ lies, for every $e$ in $E.$ This sign vector allows us to define a map 
\begin{equation}\label{eq:psi_n_map}
\psi_n\colon \U_n(\A)\to \Topes_n(\Mat(\A)),    
\end{equation}
by setting
\[
\psi_n(x^{(1)},\ldots, x^{(n)}) =  (X_x(e)\colon=(\sigma(x^{(w_x(e))})(e), w_x(e)\colon=\min\{i \mid \sigma(x^{(i)})(e)\neq 0\}).
\] 
\begin{claim*}
$\psi_n$ is a well-defined continuous map.
\end{claim*}
\begin{proof}
By definition of $\U_n(\A)$, for every $x \in \U_n(\A)$ and $e \in E$ there exists an index $i$ such that $\sigma(x^i)(e) \neq 0$, which shows that $w_x(e)$ is well-defined. Moreover, we see that we have the identity of sign vectors
\[
    X_x = \sigma(x^{(1)})\circ \cdots \circ \sigma(x^{(n)}),
\]
which proves that $X_x$ is a tope of $\om$. Hence, the map $\psi_n$ is well-defined and lands in $\Topes_n(\om)$. The preimage of lower ideal has the form
\[
\psi_n^{-1}(\downarrow (X,w_X))=\bigcap_{e\in E}\left(\bigcup_{i<w_X(e)}\left\{x \mid x^{(i)}\notin H_e\right\}\;\cup\;\left\{x \mid x^{(w(e))}\in H_e^{X(e)}\right\}\right).
\] The subsets in the union are open, as are their finite unions and intersections. Thus the preimages are open, and $\psi_n$ is continuous.
\end{proof}
 Let us state the main results of this section.

\begin{theorem}
For any pseudohyperplane arrangement $\A$ and every $n\geq1$, the map~\eqref{eq:psi_n_map} is a weak equivalence
\[
\psi_n\colon\U_n(\A)\overset{\simeq}{\longrightarrow}\Topes_n(\om(\A))
\]
\end{theorem}
\begin{proof}
The proof of this theorem follows from the results of Sections~\ref{sec:weightwise_topes},~\ref{sec:proper_topes}, and~\ref{sec:comparison-topes-prop}. Explicitly, the strategy consists of the following steps:
\begin{itemize}
    \item[(i)] We show that the subposet of weighted topes in the image of $\psi_n$, $\Prop_n(\om(\A)):=\mathrm{Im}(\psi_n)\subseteq \Topes_n(\om(\A))$, called \textit{proper topes}, is isomorphic to the Salvetti poset (Proposition~\ref{prop:prop-iso-sal}) and that $\psi_n$ induces a weak equivalence $\U_n(\A)\simeq \Prop_n(\om(\A))$ (Theorem~\ref{thm:prop-model-pseudosphere})
    \item[(ii)] We consider posets of weightwise topes $\WTopes_n(\om)$ that contain weighted topes. For the sequence of inclusions $\Prop_n(\om)\subseteq \Topes_n(\om)\subseteq \WTopes_n(\om)$, the inclusions $\Topes_n(\om)\subseteq \WTopes_n(\om)$ and $\Prop_n(\om)\subseteq \WTopes_n(\om)$ are homotopy equivalences (Theorem~\ref{thm::equivalence-Tope-Topeext} and Theorem~\ref{thm:equivalence-higher-proper-uni}). By the two-out-three property, the inclusion $\Prop_n(\om)\subseteq \Topes_n(\om)$ is a homotopy equivalence as well, which completes the proof.
\end{itemize}
\end{proof}

The following Proposition explains the relation to degenerations of pseudohyperplane arrangements.
\begin{prop}\label{prop:degeneration-compatibility}
For a degeneration $\A\rightsquigarrow\A'$ of pseudohyperplane arrangements,
let $\phi\colon\om(\A)\to\om(\A')$ be the induced weak map, and let
$s_n\colon\U_n(\A')\to\U_n(\A)$ be the morphism of
\eqref{eq::homotopy_map_complements}. The following diagram commutes in
$\mathrm{Ho}(\Top)$:
\[
\begin{tikzcd}
\U_n(\A') \arrow[r,"s_n"] \arrow[d,"\psi_n"'] &
\U_n(\A) \arrow[d,"\psi_n"] \\
{\Topes_n(\om(\A'))} \arrow[r,"\phi^*"'] &
{\Topes_n(\om(\A))}.
\end{tikzcd}
\]
\end{prop}
\begin{proof}
Consider the cylinder poset
\[
\Cyl(\phi^*)=
\Topes_n(\om(\A))\coprod\Topes_n(\om(\A')),
\]
with the original orders on the two summands and cross-relations
\[
g\leq g'\quad\Longleftrightarrow\quad g\leq\phi^*(g')
\qquad
\bigl(g\in\Topes_n(\om(\A)),\ g'\in\Topes_n(\om(\A'))\bigr).
\]
Denote the
inclusions of the generic and special summands by $j_1$ and $j_0$,
respectively. The map
\[
r\colon\Cyl(\phi^*)\longrightarrow\Topes_n(\om(\A)),
\]
given by the identity on the generic summand and by $\phi^*$ on the
special summand, is monotone and satisfies
\[
rj_1=\id,\qquad j_1r\leq\id,\qquad rj_0=\phi^*.
\]
Thus $|j_0|$ and $|r|$ are mutually inverse homotopy equivalences.\\

For $x\in\U_n(\A(t))$, write $\psi_{n,t}(x)$ for its
weighted tope computed using $\A(t)$. The trivialization of $\A(t)$ over $(0,1]$ identifies $\om(\A(t))$ with $\om(\A)$ for $t>0$. We therefore have a continuous map
\[
\pi\colon \W_n\longrightarrow\Cyl(\phi^*),\qquad
\pi(t,x)=
\begin{cases}
j_1\bigl(\psi_{n,t}(x)\bigr),&t>0,\\
j_0\bigl(\psi_{n,0}(x)\bigr),&t=0,
\end{cases}
\]
where $\W_n$ is the total higher complement defined in Section~\ref{sec:degenerations}. Let $i_1\colon\U_n(\A)\to \W_n$ and
$i_0\colon\U_n(\A')\to \W_n$ be the endpoint inclusions. By construction, we have a diagram
\[
\begin{tikzcd}[column sep=small]
\U_n(\A') \arrow[r,"i_0"] \arrow[d,"\psi_{n,0}"',"\simeq"] &
\W_n \arrow[d,"\pi"] &
\U_n(\A) \arrow[l,"i_1"',"\simeq"]
\arrow[d,"\psi_{n,1}","\simeq"'] \\
{\Topes_n(\om(\A'))} \arrow[r,"j_0"'] &
{\Cyl(\phi^*)} &
{\Topes_n(\om(\A))} \arrow[l,"j_1","\simeq"'].
\end{tikzcd}
\]
where both squares are commutative. The right square shows that $\pi$ is a homotopy equivalence, so the outer rectangle also commutes. 
\end{proof}

\subsection{Proper topes}\label{sec:proper_topes}
To characterize the image of $\psi_n$, we introduce the following definition. 
\begin{defi}
For an oriented matroid $\Mat$, we say that a weighted tope $(X,w)$ is \textit{proper} if, for each circuit $C$, we have the equality 
\[\min\{w(e)\mid C(e)X(e)=+\}=\min\{w(e) \mid C(e)X(e)=-\}.\]
We let $\Prop_n(\Mat) \subseteq \Topes_n(\Mat)$ denote the subposet of proper weighted topes. 
\end{defi}
\begin{prop}\label{prop:image-of-psi-is-proper}
For every pseudohyperplane arrangement $\A$ and every $n\geq1$, a weighted tope lies in the image of the map $\psi_n$ if and only if it is proper:
\[
\psi_n(\U_n(\A))=\Prop_n(\om(\A)).
\]
\end{prop}
\begin{proof}
Write $\om=\om(\S)$ and let $\sigma\colon\Real^r\to\La(\om)$ be the sign map of the coned arrangement, with $\sigma(0)=0$.

($\Rightarrow$) Let $x=(x^{(1)},\ldots,x^{(n)})\in\U_n(\S)$ and $\psi_n(x)=(X_x,w_x)$. For a circuit $C$, put
\[
m=\min\{w_x(e)\mid e\in\underline C\}.
\]
By definition of $w_x(e)$, we have that the support of $C$ meets the support of the covector $\sigma(x^{(m)})$. By circuit--covector orthogonality, this implies that we have two distinct elements $e \neq e' \in E$ satisfying 
\[
    C(e)\sigma(x^{(m)})(e) = + , \quad \textrm{ and } \quad C(e)\sigma(x^{(m)})(e') = -.
\]
By minimality of $m$, we must have $X_x(e) = \sigma(x^{(m)})(e)$ and $X_x(e') = \sigma(x^{m})(e')$, so we get 
\[
\{C(e)X_x(e)\mid e\in\underline C,\ w_x(e)=m\}=\{+,-\}.
\]

($\Leftarrow$) Let $(X,w)\in\Prop_n(\om)$. Define the truncated sign vectors
\[
X_i(e)=
\begin{cases}
0,&w(e) > i,\\
X(e),& w(e) \leqslant i,
\end{cases}
\qquad 1\leq i\leq n.
\]
For each circuit $C$, put $m=\min\{w(e)\mid e\in\underline C\}$. If $i<m$, then $Y_i$ vanishes on $\underline C$. If $i\geq m$, properness of $(X,w)$ supplies $e,f\in\underline C$ with $w(e)=w(f)=m$ and
\[
C(e)X_i(e)=+,\qquad C(f)X_i(f)=-.
\]
Hence $X_i$ is orthogonal to every circuit and is therefore a covector. Choose $x^{(i)}\in\Real^r$ with $\sigma(x^{(i)})=Y_i$, taking $x^{(i)}=0$ when $X_i=0$. We have $x \in \U_n(\A)$ because $x^{(n)} \in \U_1(\A)$, and we have $\psi_n(x)=(X,w)$, which is the desired result. 
\end{proof}
The next result identifies the posets of proper topes with higher Salvetti posets. 
\begin{prop}\label{prop:prop-iso-sal}
For all oriented matroid $\om$ and all $n\geqslant 1$, the posets $\Prop_n(\Mat)$ and $\Sal_n(\Mat)$ are isomorphic.
\end{prop}
\begin{proof}
Consider the map $f\colon \Prop_n(\Mat)\to \Sal_n(\Mat)$ defined by the rule
\[
(X,w)\mapsto (X_1,\cdots,X_n), \quad X_i(e)=\begin{cases} X(e),\text{ if }w(e)\leq i
\\
0, \text{ otherwise}\end{cases}
\] 
By the proof of the previous proposition (second case), each $X_i$ is a covector, so $f(X,w)$ is indeed an element of $\Sal_n(\om)$. Let us check that $f$ is a poset map. We need to check that if $(X,w_X)\leq (Y,w_{Y})$, then we have $X_{i+1}\geq X_i\circ Y_{i+1}$ for all $i$ (we use the notation $X_0=0$). Indeed, from the formulas
\[
(X_i\circ Y_{i+1})(e)=\begin{cases}
X(e),\text{ if }i\geq w_X(e)
\\
Y(e),\text{ if }w_X(e)=w_Y(e)=i+1,
\\
0, \text{ if } w_X(e)>i+1
\end{cases}, \quad X_{i+1}(e)=\begin{cases}
X(e),\text{ if }i+1\geq w_X(e)
\\
0, \text{ if } w_X(e)>i+1
\end{cases},
\] 
we see that $X_{i+1}(e)\neq (X_i\circ Y_{i+1})(e)$ only if $w_Y(e)>w_X(e)=i+1$, and in this case $(X_i\circ Y_{i+1})(e)=0$.

We consider the inverse map $g\colon \Sal_n(\Mat)\to \Prop_n(\Mat)$
\[
(X_1,\cdots,X_n)\mapsto (X_n,w),\text{ where } w(e)=\min\{i \mid  X_i(e)\neq 0\}. 
\] 
We have to check that $(X_n,w)$ is proper. For a circuit $C$, let $w_0=\min\{w(e) \mid e\in \underline{C}\}=\min \{i \mid \underline{X_i}\cap \underline{C}\neq \varnothing\}$. Since $\underline{X_{w_0}}\cap \underline{C}\neq \varnothing$, the signed set $X_{w_0}\cdot C$ is non constant. Note that, for $e\in \underline{C}$, $X_{w_0}(e)\neq 0$ only if $w(e)=w_0$, so there exists a pair $e,e'\in \underline{C}$ with $w(e)=w(e')=w_0$ and $X_n(e)C(e)\neq X_n(e')C(e')$. This implies $\phi((X_1,\cdots,X_n))\in \Prop_n(\Mat)$.
\end{proof}
We are now able to give an explicit proof of Theorem~\ref{thm:higher-sal-higher-cpt}. 
\begin{theorem}\label{thm:prop-model-pseudosphere}
Let $\A$ be a pseudohyperplane arrangement. For every $n\geq1$, the map $\psi_n$ is a weak equivalence
\[
\psi_n\colon\U_n(\A)\overset{\simeq}{\longrightarrow}\Prop_n(\om(\A))
\]
\end{theorem}
\begin{proof}
Denote by $\A=(H_e)_{e\in E}$ the coned pseudohyperplanes in $\Real^r$. By Theorem~\ref{thm:quillen-type-A}, it suffices to prove that for all $g=(X,w)\in\Prop_n(\om(\A))$, the lower fiber $U_g:=\psi_n^{-1}(\downarrow g)$ is contractible. Notice that we have 
\[
U_g
=\bigcap_{e\in E}
\left(
 \bigcup_{i<w(e)}\{y\mid y^{(i)}\notin H_e\}
 \ \cup\ \{y\mid y^{(w(e))}\in H_e^{X(e)}\}
\right),
\]
so in particular, the lower fiber $U_g$ is open in $(\Real^r)^n$. 

Choose $x\in\U_n(\A)$ with $\psi_n(x)=g$, using Proposition~\ref{prop:image-of-psi-is-proper}. Following the coordinate-by-coordinate strategy of~\cite[Theorem~5.8]{beuckelmann2022small}, set
\[
Z_k=\{y\in U_g\mid y^{(i)}=x^{(i)}\text{ for }i>k\},
\qquad 0\leq k\leq n.
\]
Then $Z_n=U_g$ and $Z_0=\{x\}$. We show that replacing the $k$-th coordinate by $x^{(k)}$ defines a homotopy equivalence
\[
p_k\colon Z_k\longrightarrow Z_{k-1}.
\]
We start by checking that $p_k$ is well-defined. Let $y\in Z_k$ and $z = p_k(y) = (y^{(1)}, \ldots , y^{(k-1)}, x^{(k)}, \ldots , x^{(n)}).$ For any $e\in E$, if $w_y(e) < k$ then we must have $w_z(e) = w_y(e)$ and $X_z(e) = X_y(e)$. If $w_y(e) \geqslant k$ then we have $w_z(e) \leqslant w_y(e)$ and $X_z(e) = X_x(e).$ In both cases we see that $(X_z(e), w_z(e))$ satisfies the right condition for all $e$, so that $z$ belongs to $Z_{k-1}.$

We now prove that $p_k$ is a homotopy equivalence. We prove this using chamber contractibility and finite open covers. Put $E_k=\{e\in E\mid w_x(e)=k\}$. Identifying $Z_{k-1}$ with its first $k-1$ coordinates, we have the identification
\[
Z_k=
\left\{(b,z)\in Z_{k-1}\times\Real^r\ \middle|\
 \forall e\in E_k,\ 
 \left(\exists i<k:\ b^{(i)}\notin H_e\right)
 \text{ or }z\in H_e^{X_x(e)}
\right\}.
\]
Under this identification, $p_k$ is the projection onto $B$.

For $I\subseteq E_k$, define the open sets
\[
O_I=\{b\in B\mid
 \forall e\in E_k\setminus I,\ \exists i<k:\ b^{(i)}\notin H_e\},
\qquad
V_I=\bigcap_{e\in I}H_e^{X(e)},
\]
with $V_\varnothing=\Real^r$. Every $V_I$ contains $x^{(k)}$ and is contractible. Indeed, when $I\ne\varnothing$, its radial section
\[
V_I\cap\Sph^{r-1}=\bigcap_{e\in I}S_e^{X(e)}
\]
is a nonempty chamber of the subarrangement $(S_e)_{e\in I}$, hence an open ball by the chamber property of pseudosphere arrangements~\cite[Lemma~5.1.8]{bjorner1999oriented}. 

The families $(O_I\times V_I)_{I\subseteq E_k}$ and $(O_I)_{I\subseteq E_k}$ are finite open covers of $Z_k$ and $Z_{k-1}$, respectively; in particular, $O_{E_k}=Z_{k-1}$. For any nonempty finite family $I_1,\ldots,I_m\subseteq E_k$, the corresponding intersection map is
\[
\bigcap_{j=1}^m(O_{I_j}\times V_{I_j})
=
O_{\bigcap_{j=1}^m I_j}\times V_{\bigcup_{j=1}^m I_j}
\ \longrightarrow\
O_{\bigcap_{j=1}^m I_j}
=
\bigcap_{j=1}^m O_{I_j}.
\]
It is a homotopy equivalence because its second factor is contractible. The open-cover gluing lemma for continuous maps now implies that $p_k$ is a homotopy equivalence. Consequently
\[
U_g=Z_n\simeq Z_{n-1}\simeq\cdots\simeq Z_0=\{x\},
\]
and all the lower fibers of $\psi_n$ are contractible. 
\end{proof}
Our next goal is to show that the posets of weighted topes also model higher arrangement complements. In order to achieve this, we introduce yet another new poset model, which will play an auxiliary role. 

\subsection{Weightwise topes and comparison with weighted topes}\label{sec:weightwise_topes}
\begin{definition}
    For an oriented matroid $\Mat$, we say that a weighted signed set $(X,w)$ with $\underline{X}=E$ is a \emph{weightwise tope} if $X$ contains no circuit of constant weight. We denote by $\WTopes_n(\om)$ the poset of weightwise topes with the order given by \eqref{eq:def-order-weighted-tope}.
\end{definition} 
In particular, $\WTopes_n(\om)$ is a subposet of $\Topes_n(\om).$ Similarly to weighted topes, the posets of weightwise topes carry an operadic structure. 
\begin{prop}
   For all $n\geqslant 1$, the family of posets $\WTopes_n$ carries a presheaf structure defined by the rule 
    \[
        \WTopes_n(\phi)\colon \WTopes_n(\N)\to \WTopes_n(\om), \quad (X, w) \mapsto (\epsilon (X\circ f), w\circ f),
    \]
for every morphism $\phi = (f, \epsilon)\colon \om \rightarrow \N$.
\end{prop}
\begin{proof}
    Let $(X,w) \in \WTopes_n(\N)$ be a weightwise tope. If the weighted sign vector $(\epsilon (X\circ f), w\circ f)$ contains a signed circuit $C$ of constant weight, then the signed multiset $f(C)$ contains a signed circuit of $\N$ of constant weight, and so $(X, w)$ is not a weightwise tope, a contradiction. Monotonicity is immediate. 
\end{proof}

The main result of this section is the following theorem. 

\begin{theorem}\label{thm::equivalence-Tope-Topeext} 
    For any oriented matroid $\om$ and any $n\geq 1$, the inclusion 
    \[
        j:\Topes_n(\Mat)\hookrightarrow\WTopes_n(\Mat)
    \]
    is a homotopy equivalence.
\end{theorem}
The proof follows the same pattern as in~\cite[Thm.~4.5]{beuckelmann2022small}. We split the proof into two lemmas, which will be of independent interest. For the proof, we will use the equivalent \emph{strong circuit elimination} axiom: for $C,D\in\C$, $e\in S(C,D)$ and $f\in\supp D\setminus\supp C$, there exists $Z\in\C$ such that
\begin{equation}\label{eq:strong-elimination}
 Z^\pm\subseteq(C^\pm\cup D^\pm)\setminus\{e\},
 \qquad f\in\supp Z.
\end{equation}
The following consequence is the form needed in this article.

\begin{lemma}[{\cite[Propositions 3.1 and 3.4]{Fukuda2004}}]\label{lem:conformal-elimination}
Let $C$ and $D$ be circuits of an oriented matroid $\om$. If $S(C,D)$ is nonempty
and $f\in\supp D\setminus\supp C$, then there exist $e\in S(C,D)$ and a
circuit $Z$ such that
\begin{equation*}
 f\in\supp Z,\qquad Z\preceq C\circ D,\qquad \textrm{and } \quad  Z(e)=0.
\end{equation*}
\end{lemma}

\begin{lemma}\label{lem:circuit-descent-uniform}
Let $\om$ be loopless and $g=(X,w)\in\WTopes_n(\om)$. Suppose $X$ is not a tope. Choose $C\preceq X$ minimizing
lexicographically the tuple
\[
 \mu_w(C)=\bigl(|\supp C\cap w^{-1}(m)|,\ldots,
                    |\supp C\cap w^{-1}(1)|\bigr).
\]
For any $S\subseteq\supp C\cap\{w>1\}$, define
\[
 g_S=(X_S,w_S),\qquad
 w_S=w-\mathbf 1_S,\qquad
 X_S(f)=\begin{cases}-X(f)&f\in S,\\ X(f)&f\notin S.\end{cases}
\]
Then the weighted sign vector $g_S$ belongs to $\WTopes_n(\om)$.
\end{lemma}
\begin{proof}
The case $S=\varnothing$ is immediate. Suppose, for a contradiction, that a circuit
$D\preceq X_S$ has constant new weight $k$. Set
\[
 A_0=\supp D\cap S,\qquad B_0=\supp D\setminus S.
\]
We have $A_0\ne\varnothing$, since otherwise $D$ already violates
uniformity of $g$. On $A_0$, the old weight is $k+1$ and $D=-C$;
on $B_0$, the old weight is $k$ and $D=X$. Hence $S(C,D)=A_0$.
We have $D\ne C$. Also $D=-C$ would force $\supp C\subseteq S$,
making $C\preceq X$ a circuit of constant old weight $k+1$, which is
impossible. Circuit incomparability therefore supplies
$f\in\supp D\setminus\supp C$.

Apply Lemma \ref{lem:conformal-elimination} to eliminate $e$ and keep $f$. The resulting circuit $E$ satisfies $E \preceq X$.
Every coordinate added to $\supp C$ has old weight $k$, while at least
one coordinate of old weight $k+1$ is removed. No higher-weight coordinate
is added, so $\mu_w(E)<\mu_w(C)$, contrary to the choice of $C$. This finishes the proof. 
\end{proof}
\begin{lemma}\label{lem:contractible-fibers-uniform}
For every $g\in\WTopes_n(\om)$, the poset
\[
 \Topes_n(\om)\cap\downarrow g
\]
is contractible.
\end{lemma}
\begin{proof}
Write $g = (X, w) \in \WTopes_n(\om)$. If $X$ is a tope, the poset $ \Topes_n(\om)\cap\downarrow g$ has an upper bound given by $g$, and so it is contractible. \\

Otherwise, choose $C$ as in Lemma \ref{lem:circuit-descent-uniform} and put $J=\supp C\cap\{w>1\}$. 
This set is nonempty by uniformity. Every tope $Y$ disagrees with $C$
on some coordinate of its support. If $(Y,v)\leq g$, that coordinate
lies in $J$ and $(Y,v)\leq u_{\{s\}}$. Consequently
\[
 \Topes_n(\om)\cap\downarrow g
 =\bigcup_{s\in J}\bigl(\Topes_n(\om)\cap\downarrow g_{\{s\}}\bigr).
\]
The intersection indexed by a nonempty $S\subseteq J$ is exactly
$\Topes_n(\om)\cap\downarrow g_S$. By Lemma \ref{lem:circuit-descent-uniform} and the induction hypothesis,
this intersection is nonempty and contractible. Theorem~\ref{thm:cover-contractible} concludes the proof. 
\end{proof}
\begin{proof}[Proof of Theorem \ref{thm::equivalence-Tope-Topeext}] 
Lemma \ref{lem:contractible-fibers-uniform} shows that the lower fiber 
\[
    j^{-1}(\downarrow g) = \Topes_m(L)\cap\downarrow g
\]
is contractible for every $g \in \WTopes_n(\om)$. By Quillen's Theorem A~\cite{quillen2006higher}, this shows that $j$ is a homotopy equivalence.

\end{proof} 
	
\subsection{Comparison between $\Topes_n$ and $\WTopes_n$}\label{sec:comparison-topes-prop}
In this section we show that the posets of weightwise topes $\WTopes_n$ form a model of higher arrangement complements, and so do the posets of weighted topes by Theorem~\ref{thm::equivalence-Tope-Topeext}.
\begin{theorem}\label{thm:equivalence-higher-proper-uni}
Let $\om$ be an oriented matroid and $n \geqslant 2$ an ingeger. The inclusion 
\[
\Prop_n(\om) \subseteq \WTopes_n(\om)
\]
is a homotopy equivalence.
\end{theorem}
\begin{proof}
    The proof follows the argument of Theorem \ref{thm::equivalence-Tope-Topeext}. If $(X,w)$ is a weighted sign vector and $i \geq 1$ is an integer, we denote by $X_i$ the sign vector defined by
    \[
        X_i(e) = 
        \left\{\begin{array}{cl}
            X(e) &\textrm{ if } w(e) \leq i, \\
            0 &\textrm{ otherwise.}
        \end{array}\right.
    \]
    By~Proposition~\ref{prop:prop-iso-sal}, the elements of $\Prop_n(\om)$ are exactly the weighted sign vectors $(X, w)$ such that $X_i$ is a covector of $\om$ for all $i\geq 1.$ For the purpose of this proof only, we define auxiliary subposets $P_i \subseteq \WTopes_n(\om)$ for all $i$ between $1$ and $n$, by
    \[
    P_i \coloneqq \{ (X,w) \in \WTopes_n(\om) \mid X_i, \ldots,X_n \textrm{ are covectors} \}.
    \]
    This gives a sequence of posets 
    \[
        P_1 = \Prop_n(\om) \subseteq \cdots \subseteq P_n = \Topes_n(\om) \subseteq \WTopes_n(\om).
    \]
    We prove Theorem \ref{thm:equivalence-higher-proper-uni} by showing that all the lower fibers $\iota/g$ of the inclusion 
    \[
        \iota: \Prop_n(\om) \hookrightarrow \WTopes_n(\om)
    \]
    are contractible, by induction on $g$ in the finite poset $\WTopes_n(\om)$. If $g\in\Prop_n(\om)$, then $\iota/g$ has a greatest element $g$, so this lower fiber is contractible. If $g\notin\Topes_n(\om)$, consider the inclusion
    \[
        \iota/g \hookrightarrow \Topes_n(\om)\cap\downarrow g.
    \]
    Its lower fiber at $h\in\Topes_n(\om)\cap\downarrow g$ is $\iota/h$, which is contractible by induction since $h<g$. The target is contractible by Lemma~\ref{lem:contractible-fibers-uniform}, so Theorem~\ref{thm:quillen-type-A} shows that $\iota/g$ is contractible.

    We may therefore assume that $g=(X,w)\in\Topes_n(\om)\setminus\Prop_n(\om)$. Choose $1\leq i<n$ such that $g\in P_{i+1}\setminus P_i$. Since $X_i$ is not a covector, there exists a signed circuit $C_0$ such that $X_i$ is not orthogonal to $C_0$. Up to passing to $-C_0$, we have
    \[
        \forall e \in \underline{X_i} \cap \underline{C_0} \quad C_0(e) = X_i(e) \quad \textrm{ and } \quad \underline{X_i} \cap \underline{C_0} \neq \emptyset.
    \]
    Choose such a $C_0$ minimizing $\mu_w$ as in Lemma~\ref{lem:circuit-descent-uniform}, and put
    \[
        H=\{e\in E\mid w(e)>i\},\qquad
        J=\underline{C_0}\cap\{e\in E\mid w(e)>1\}.
    \]
    Since $X$ is a tope, $C_0$ also meets $H$, so $J$ is nonempty. For every $S\subseteq J$, define $g_S=(X_S,w_S)$ by
    \[
        w_S=w-\mathbf 1_S,\qquad
        X_S(e)=\begin{cases}
            -C_0(e),&e\in S,\\
            X(e),&e\notin S.
        \end{cases}
    \]
    Thus $g_S\leq g$, with strict inequality when $S$ is nonempty. We have the following key claim. \\ \\
    \textbf{Claim.} For every $S\subseteq J$, the weighted sign vector $g_S$ is a weightwise tope.
    \begin{proof}[Proof of Claim.]
        The set $\underline{C_0}\cap H$ is independent, being a proper subset of the support of a circuit. Hence $C_0|_H$ extends to a tope $T$ of $\om|_H$. Let $\widetilde X$ agree with $X$ outside $H$ and with $T$ on $H$. Then $\widetilde g=(\widetilde X,w)$ is weightwise: a circuit of constant weight would be contained either in $X$ or in $T$. Moreover, $C_0\preceq\widetilde X$, and every circuit contained in $\widetilde X$ meets $E\setminus H$ and is therefore nonorthogonal to $X_i$. Thus $C_0$ is still $\mu_w$-minimal among these circuits. Lemma~\ref{lem:circuit-descent-uniform} shows that $\widetilde g_S=(\widetilde X_S,w_S)$ is a weightwise tope, where $\widetilde X_S$ equals $-C_0$ on $S$ and $\widetilde X$ elsewhere.

        Suppose that $D\preceq X_S$ is a circuit of constant new weight $k$. If $k\leq i$, then $D\preceq\widetilde X_S$: on $S$ the two sign vectors agree, and outside $S$ every element of $D$ has old weight $k\leq i$. This contradicts the preceding application of Lemma~\ref{lem:circuit-descent-uniform}.

        If $k>i$, then $\underline D\setminus S$ is nonempty. Indeed, otherwise circuit incomparability would give $\underline D=\underline{C_0}$, contrary to the fact that $C_0$ meets $\{w\leq i\}$. On $\underline D\setminus S$ the old weight is $k$ and $D$ agrees with $X_k$, whereas on $\underline D\cap S$ the old weight is $k+1$ and $X_k$ vanishes. Thus $D$ is not orthogonal to the covector $X_k$, a contradiction.
    \end{proof}
    For $p=(Y,v)\in\iota/g$, properness supplies an element $s\in\underline{C_0}$ with $Y(s)=-C_0(s)$ and $v(s)=\min_{e\in\underline{C_0}}v(e)\leq i$. Since $C_0$ agrees with $X$ on $\underline{C_0}\cap\{w\leq i\}$ and $p\leq g$, we have $v(s)<w(s)$. Hence $s\in J$ and $p\leq g_{\{s\}}$. This shows that the lower ideals $\iota/g_{\{s\}}$, $s\in J$, cover $\iota/g$. Theorem~\ref{thm:cover-contractible} concludes the proof.
\end{proof}
We have the following immediate corollary of Theorem~\ref{thm::equivalence-Tope-Topeext} and Theorem~\ref{thm:equivalence-higher-proper-uni}. 
\begin{corollary}\label{coro:prop-tope-equivalence}
    Let $\om$ be an oriented matroid, and $n\geqslant 2$ an integer. The inclusion 
    \[
        \Prop_n(\om) \subseteq \Topes_n(\om)
    \]
    is a homotopy equivalence. 
\end{corollary}

\subsection{The Delucchi-Falk model} 
In~\cite{DelucchiFalk2017}, Delucchi and Falk introduced another poset model of pairs of topes that is equivalent to the $2$-dimensioanl Salvetti poset. We show that these posets form an $\OM$-presheaf as well and that this presheaf is homotopy equivalent to $\Topes_2$

\begin{definition} 
For an oriented matroid $\om$, the Delucchi-Falk poset $\mathsf{DF}(\om)$ is the poset with underlying set $\Topes(\om)\times\Topes(\om)$ and with the order
\[
(X_1;X_2)\leq (X'_1,X'_2):\Leftrightarrow S(X_1;X'_1)\subseteq S(X_2;X'_1)\subseteq S(X'_2;X'_1)
\] where $S(X,Y)=\{e\in \supp X\cap \supp Y| X(e)\neq Y(e)\}$. Equivalently, the order is given by 
\[
(X_1;X_2)\leq (X'_1,X'_2):\Leftrightarrow  X_1\in [X'_1,X_2]\text{ and }X_2\in [X'_1,X'_2],
\] where $A\in [B,C]$ means that, for all $e$, $A(e)\in \{B(e),C(e)\}$.
\end{definition}

Note that these posets form an $\OM$-presheaf in posets $\DF\colon \OM^{\op}\to \Pos$ defined by
\[
\phi^*\colon \DF(\N)\to \DF(\om), \quad (X,Y)\mapsto (\phi^*X,\phi^*Y),
\] We only have to check that $\phi^*$ is order preserving. It follows from the identity
\[
S(\phi^*X,\phi^*Y)=f_{\phi}^{-1}(S(X,Y)).
\]
\begin{theorem} We have an equivalence of
$\OM$-presheaves
\[
\Phi\colon \DF\overset{\simeq}{\to} \Topes_2
\] defined by
\[
(X,Y)\mapsto (X,w_{XY}), \quad w_{XY}(e)=\begin{cases}
1,\text{ if }X(e)=Y(e)
\\
2,\text{ if }X(e)\neq Y(e)
\end{cases}
\]
\end{theorem}
\begin{proof}
Let us firstly check that $\Psi_{\om}\colon \DF(\om)\to \Topes_2(\om)$ is order-preserving
\[
(X_1,Y_1)\leq (X_2,Y_2)\Rightarrow (X_1,w_{X_1Y_1})\leq (X_2,w_{X_2Y_2})
\] We check it coordinatewise: if $X_1(e)\neq X_2(e)$, the relation $(X_1,Y_1)\leq (X_2,Y_2)$ implies that $X_1(e)=Y_1(e)$ and $X_2(e)\neq Y_2(e)$, hence $w_{X_1Y_1}(e)=1<w_{X_2Y_2}(e)=2$. If $X_1(e)\neq X_2(e)$, the inequality $X_1(e)\neq Y_1(e)$ always leads to $X_2(e)\neq Y_2(e)$, which implies $w_{X_1Y_1}(e)\leq w_{X_2Y_2}(e)$. 

Compatibility of $\Psi$ with pullbacks follows from
\[
\Phi_{\om}(\phi^*(X,Y))=\Phi_{\om}((\phi^*X,\phi^*Y))=(\phi^*X,w_{\phi^*X\phi^*Y})=(\phi^*X,f^*_{\phi}w_{XY})=\phi^*\Phi_{\N}((X,Y)).
\] So, $\Psi$ is a well-defined morphism of poset presheaves.

It remains to check that $\Psi_{\om}$ is a homotopy equivalence for every $\om$. Recall that there is a homotopy equivalence~\cite[Lemma~3.12]{DelucchiFalk2017}
\[
\Sal_2(\om)\overset{\simeq}{\to} \mathsf{DF}(\om), \quad (T,C)\mapsto (T,C\circ T^{\op})
\] Note that it fits into the following commutative triangle
\[\begin{tikzcd}
	{\DF(\om)} && {\Topes_2(\om)} \\
	& {\Sal(\om)}
	\arrow["{\Psi_{\om}}", from=1-1, to=1-3]
	\arrow["\simeq", from=2-2, to=1-1]
	\arrow["\simeq"', from=2-2, to=1-3]
\end{tikzcd}\] By the two-out-three property, $\Psi_{\om}$ is an equivalence.
\end{proof}

    Recall that there is an anti-automorphism $\rho\colon \mathsf{DF}(\om)\to \mathsf{DF}(\om)^{\op}, \quad (X_1,X_2)\mapsto (-X_2,X_1)$ of order $4$~\cite[Lemma~3.13]{DelucchiFalk2017}. It defines a presheaf anti-automorphism 
    \[
    \mathsf{DF}\overset{\cong}{\to} \mathsf{DF}^{\op}.
    \]The corresponding $\Z_4$-action on $\Delta(\mathsf{DF})$ corresponds to the discrete subgroup $\{\pm1,\pm i\}\subseteq S^1$.

\section{Combinatorial properties of $\Topes_n$ and $\WTopes_n$}
In this section we initiate the combinatorial study of the new poset models that were introduced in Section~\ref{sec:new-presheaf-models}. 
\subsection{Purity and rank of $\Topes_n$ and $\WTopes_n$}
Recall that a finite poset $\Po$ is called \emph{pure} if all the maximal chains of $\Po$ have the same cardinality. This is equivalent to the order complex $\Delta(\Po)$ being pure. If $P$ is pure, we define the rank of $\Po$, denoted by $\rk \Po$, to be the cardinality of any maximal chain of $\Po$ minus one. This is also the dimension of $\Delta(\Po)$. It is a well-known fact that for every oriented matroid $\om$ and every integer $n \geq 2$, the poset of proper topes $\Prop_n(\om)$ is pure of rank $\rk \om$. In this section we study the purity and rank of the posets of weighted and weightwise topes. The main results in this direction are the following propositions. 
\begin{prop}\label{prop:purity-Topes}
Let $\om$ be an oriented matroid, and $n\geq 2$ an integer. The poset $\Topes_n(\om)$ is pure if and only if $\om$ is simple. In that case, the rank of $\Topes_n(\om)$ is $(n-1)\cdot E(\om).$
\end{prop}
\begin{proof}
    If $\om$ admits two parallel or antiparallel elements $e$ and $f$, then, for any tope $T$ of $\om$, let us denote by $p$ the element of $\Topes_n(\om)$ with tope $T$ and constant weight $1,$ $q$ the element of $\Topes_n(\om)$ with tope $T$ and constant weight equal to $d$, and $q^{-}$ the element of $\Topes_n(\om)$ with tope $-T$ and constant weight $d.$ There is a maximal chain from $p$ to $q$ of size $(n-1)\cdot \#E(\om)+1 $, given by raising the weights of $p$ one by one, without changing any sign. On the other hand, one can see that there is no maximal chain between $p$ and $q^{-}$ having size $(n-1)\cdot \#E(\om)+1$, because as soon as one changes the sign of $e$, one must also change the sign of $f$ in the same step. Thus, in this case, $\Topes_n(\om)$ is not pure. \\ 

    Let us now assume that $\om$ is simple, and consider two elements $(T, w) \lessdot (T', w')$ of $\Topes_n(\om)$ forming a covering relation in $\Topes_n(\om)$. We will show that we have the equality 
    \begin{equation}\label{eq:covering-relation-weight}
        \sum_{e\in E(\om)} (w'(e) - w(e)) = 1. 
    \end{equation}
    Indeed, note that there cannot be any element $e \in E(\om)$ such that $w'(e) - w(e) > 1,$ as otherwise the element obtained from $(T,w)$ by raising the weight of $e$ by one and keeping the same signs would lie strictly between $(T,w)$ and $(T', w')$. In addition, if there are two distinct elements $e\neq f$ such that we have both $w'(e)- w(e) > 0$ and $w'(f) - w(f) >0$, then we see that $T(e)$ must be different from $T'(e)$, as otherwise the element obtained from $(T, w)$ by raising only the weight at $e$ by one would lie strictly between $(T,w)$ and $(T, w')$. The same goes for the sign at $f$, so the topes $T$ and $T'$ differ on both $e$ and $f$. This means that the Hamming distance between $T$ and $T'$
    \begin{equation*}
     \#\{ e \in E(\om) \mid T(e) \neq T'(e)\}
    \end{equation*}
    is at least two. By \cite[Proposition 4.2.3]{bjorner1999oriented} and the simplicity of $\om$, this implies that the distance between $T$ and $T'$ in the tope graph of $\om$ is also at least two. One can then use a tope between $T$ and $T'$ in that tope graph to produce an element lying strictly between $(T,w)$ and $(T',w')$, which is again a contradiction. Hence, we have the claimed equality \eqref{eq:covering-relation-weight}. 
    
    Finally, notice that the maximal elements of $\Topes_n(\om)$ have constant weight equal to $d$ and the minimal elements have constant weight equal to $1$. This fact, together with \eqref{eq:covering-relation-weight}, shows that the maximal chains of $\Topes_n(\om)$ have size $(d-1)\cdot \#E(\om) + 1,$ which is the desired result. 
\end{proof}
\begin{prop}\label{prop:purity-Uni}
    For every oriented matroid $\om$ and all $n\geq 2$, the poset $\WTopes_n(\om)$ is pure of rank $(n-1)\cdot \#E(\om).$ 
\end{prop}
\begin{proof}
    By the proof of Theorem \ref{thm::equivalence-Tope-Topeext}, the minima of $\WTopes_n(\om)$ are the elements which have constant weight $1$, and the maxima of $\WTopes_n(\om)$ are the elements which have constant weight $n$. It remains to determine the covering relations of $\WTopes_n(\om)$. Let $(T, w) \lessdot (T', w')$ be two elements of $\WTopes_n(\om)$ forming a covering relation. As in the proof of Proposition \ref{prop:purity-Topes}, we will show that we have the equality 
    \begin{equation}\label{eq:covering-relation-weight-uniform}
        \sum_{e\in E(\om)} (w'(e) - w(e)) = 1. 
    \end{equation}
    The proof is slightly different from the case of $\Topes_n$. First, suppose for a contradiction that there exists an element $e$ such that we have $w'(e) - w(e) > 1$. Then, there exists a sign $\epsilon$ such that the weighted sign vector $(T^{\epsilon}, w^{e})$ defined by 
    \begin{equation*}
        T^{\epsilon}(f)= \begin{cases} T(f) \textrm{ if } f \neq e, \\ \epsilon \textrm{ otherwise,} \end{cases}  \quad \textrm{ and } \quad 
        w^{e}(f) = \begin{cases}
            w(f) \textrm{ if } f \neq e, \\ w(e) + 1 \textrm{ otherwise,}
        \end{cases}
    \end{equation*}
    belongs to $\WTopes_n(\om).$ Indeed, if this were false, we would have two circuits $C_1,C_2$ contained respectively in $T^{+}$ and $T^{-}$, of constant weight $w(e)+1.$ Applying circuit elimination to theses two circuits would then contradict the fact that $(T, w)$ is a weightwise tope. It then appears that $(T^{\epsilon}, w^{e})$ lies strictly between $(T, w)$ and $(T', w')$, which is a contradiction. 

    Next let us consider an element $e$ such that we have $w'(e) = w(e) + 1$, and which has maximal $w$-weight amongst such elements. The element $(T^{e}, w^{e})$ defined by 
     \begin{equation*}
        T^{e}(f)= \begin{cases} T(f) \textrm{ if } f \neq e, \\ T'(f) \textrm{ otherwise,} \end{cases}  \quad \textrm{ and } \quad 
        w^{e}(f) = \begin{cases}
            w(f) \textrm{ if } f \neq e, \\ w(e) + 1 \textrm{ otherwise,}
        \end{cases}
    \end{equation*}
    belongs to $\WTopes_n(\om)$. Indeed, $T^{e}$ does not contain any circuit of constant weight different from $w(e) + 1$ because $(T, w)$ is a weightwise tope, and $T^e$ does not contain any circuit of constant weight $w(e) +1$ because $(T',w')$ is a weightwise tope. We have the inequalities
    \[
        (T,w) < (T^{e},w^{e}) \leq (T',w'),
    \]
    which implies that we have $(T^{e},w^{e}) = (T',w')$. This concludes the proof of \eqref{eq:covering-relation-weight-uniform}. By equation \eqref{eq:covering-relation-weight-uniform} and the description of the extrema of $\WTopes_n(\om)$, it follows that all the maximal chains of $\WTopes_n(\om)$ have cardinality $(n-1)\cdot \#E(\om).$
\end{proof}
\subsection{The CW property for $\Topes_n$ and $\WTopes_n$}
Recall that a finite poset $\Po$ is called a \emph{CW poset} if it is the poset of cells of a regular CW complex. By a classical result of Bjorner~\cite[Proposition 3.1]{Bjorner1984}, a poset $\Po$ is a CW poset if and only if it is pure and for every element $p\in \Po$ the strict lower ideal 
\[
    \Po_{<p} \coloneqq \{ q \in \Po \mid q < p\} 
\]
has the homotopy type of a sphere of dimension $\rk p - 1$. We have the following classical result. 

\begin{prop}[{\cite{BjornerZiegler1992}}]\label{prop:proper-topes-is-CW}
    For every oriented matroid $\om$ and every integer $n\geq 2$, the poset $\Prop_n(\om)$ is a CW poset. 
\end{prop}
\begin{proof}
    For $n=2$, this is a consequence of \cite[Corollary 8.11]{BjornerZiegler1992}, which is proved in the more general setting of $2$-pseudoarrangements. As described in \cite[Section 9]{BjornerZiegler1992}, this result can be further generalized to $d$-pseudoarrangements for all $n$, which gives the result for $n>2.$
\end{proof}
The next proposition shows that, in contrast, the posets of weighted topes and weightwise topes are almost never CW posets. 
\begin{prop}
    Let $\om$ be an oriented matroid and $n\geq 2$ an integer. The poset $\Topes_n(\om)$ is a CW poset if and only if $\om$ is a free oriented matroid. The same statement is true for $\WTopes_n(\om).$
\end{prop}
\begin{proof}
    If $\om$ is Boolean, then we have the equalities of posets 
    \[
        \Prop_n(\om) = \Topes_n(\om) = \WTopes_n(\om),
    \]
    so the result follows from Proposition \ref{prop:proper-topes-is-CW}. 

    On the other hand, if $\om$ is not Boolean, then one can find two sign vectors $X$ and $T$ such that $T$ is a tope, $X$ is not a tope, and $X$ and $T$ differ at a single element $e$. Then, one can consider the weighted sign vector $(T,w)$, where the weight $w$ is defined by 
    \[
        w(f) = \begin{cases}
            1 \textrm{ if } f \neq e, \\
            2 \textrm{ otherwise.}
        \end{cases}
    \]
    This defines an element of $\Topes_n(\om) \subseteq \WTopes_n(\om)$. For $P = \Topes_n(\om), \WTopes_n(\om)$, the subposet $P_{<(T,w)}$ contains a single element which is the weighted tope $(X, \mathsf{1})$. Hence, the subposet $P_{<(T,w)}$ does not have the homotopy type of a sphere, neither $ \Topes_n(\om)$ nor $\WTopes_n(\om)$ is a CW poset.  
\end{proof}
\section{Deletion-Contraction and Cohomology rings}\label{sec:deletion-contraction-cohomology-ring}
In this section we study the cohomology rings of the posets of weighted topes and the copresheaf structure induced by the presheaf structure on $\Topes_n$. More precisely, we show (Theorem~\ref{thm::pres_H(Topes)}) that we have an isomorphism of copresheaves in graded commutative algebras over $\Z$
\[
    H^{\bullet}(\Topes_n) \simeq \OS_n^{\bullet},
\]
where the copresheaf $\OS^{\bullet}_n$ has been defined in Proposition~\ref{prop:presheaf-Orlik-Solomon}. If we restrict to the algebra structure, this generalizes Theorem~\ref{thm:cohomology-arrangement-OS} to pseudosphere arrangements, which seems to be a folklore result.

Our strategy for computing the cohomology ring of $\Topes_n(\om)$ for arbitrary oriented matroids is to set-up a geometric deletion-contraction argument in a purely combinatorial setting. This argument goes as follows.

Assume $\A$ is a real hyperplane arrangement, and $H$ is a hyperplane of $\A.$ The intersection $\U_2(\A\setminus H)\cap H$ is a subvariety of $\U_2(\A\setminus H)$ which can be identified with the complexified arrangement complement of the restriction $\A^H$. Consider a tubular neighbourhood $N$ of $\U_2(\A\setminus H)\cap H \subseteq \U_2(\A\setminus H)$. The boundary $\partial N_1$ at radius $1$ is homeomorphic to $\U_2(\A^{H})\times S^1$. Finally, the complement $\U_2(\A\setminus H) \setminus N_{<1}$ is homotopy equivalent to $\U_2(\A)$. We can glue $N_{<1}$ and $\U_2(\A\setminus H) \setminus N_{<1}$ along $\partial N_1$ to obtain $\U_2(\A\setminus H).$ This gives the pushout square of spaces 
 \begin{equation*}
        \begin{tikzcd}
        \U_2(\A^H)\times S^1 \arrow[r] \arrow[d] & \U_2(\A)\arrow[d]\\ 
        \U_2(\A^H) \arrow[r] & \U_2(\A\setminus H). 
        \end{tikzcd}
\end{equation*}
Since the inclusion $\partial N_1 \hookrightarrow N_{\leq 1}$ is a cofibration, the above diagram is a homotopy pushout square.

\subsection{Deletion-contraction pushout for $\Topes_n$}

We start by recalling how to compute homotopy pushouts of posets. We say that a diagram of posets is a homotopy pushout square if the corresponding diagram of topological spaces is a homotopy pushout square
    \begin{equation*}
        \begin{tikzcd}
        {\Po} \arrow[r,"f"] \arrow[d, "g"] & {\Qo}\arrow[d]\\ 
        {\mathsf{R}} \arrow[r] & {\mathsf{S}} \arrow[ul, phantom,
        "\overset{h}{\ulcorner}"{description, pos=0.25}] 
        \end{tikzcd}\text{ is a homotopy pushout if }
         \begin{tikzcd}
        {\vert\Po\vert} \arrow[r,"|f|"] \arrow[d, "|g|"] & {\vert\Qo\vert}\arrow[d]\\ 
        {\vert\mathsf{R}\vert} \arrow[r] & {\vert\mathsf{S}\vert} \arrow[ul, phantom,
        "\overset{h}{\ulcorner}"{description, pos=0.25}] 
        \end{tikzcd} \text{ is a homotopy pushout in }\Top
    \end{equation*}
Given a span $\mathsf{R}\overset{g}{\leftarrow}\Po\overset{f}{\to}\Qo$,
one can construct an explicit model for a homotopy pushout $\mathsf{R} \cup^h_{\Po} \Qo$~\cite[Corollary 3.3]{FernandezMinian2016}. The underlying set of $\mathsf{R} \cup^h_{\Po} \Qo$ is $\mathsf{R} \sqcup \Po \sqcup \Qo$ and the order relations are given by the order relations from $\mathsf{R}$, $\Po$ and $\Qo$, together with the cross relations
\[
    p \leqslant r \quad \forall p,r \textrm{ s.t. } f(p) \leqslant r,  \quad \textrm{ and } \quad  p \leqslant q \quad \forall q,r \textrm{ s.t. } g(p) \leqslant r.
\]
Let us consider the deletion-contraction pullback square in $\OM$ from Example~\ref{ex:deletion_contraction_OM}
\begin{equation}
    \begin{tikzcd}
    	{\om\setminus e} & \om \\
    	{\om/e} & {\om/e\oplus\om|_e}
    	\arrow[from=1-1, to=1-2]
    	\arrow[from=1-1, to=2-1]
    	\arrow["\lrcorner"{anchor=center, pos=0.125}, draw=none, from=1-1, to=2-2]
    	\arrow[from=1-2, to=2-2]
    	\arrow[from=2-1, to=2-2]
        \end{tikzcd}
\end{equation} Recall that $\Topes$ sends this pullback square to a pushout square of sets. The following theorem is the higher derived version of this statement

\begin{theorem}[Deletion-contraction homotopy pushout square]\label{thm:deletion-contraction-pushout}
Let $n\geq 2$ be an integer. For every oriented matroid $\om$ and a flat $e\in E(\om)$, the square of posets
\[
\begin{tikzcd}
\Topes_n(\om/e)\times \Topes_n(\om|_e) \arrow[r, "\mu_e"] \arrow[d,"\pr_1"] & \Topes_n(\om) \arrow[d, "\res_{E\setminus e}"]\\ 
\Topes_n(\om/e) \arrow[r, "\res_{\backslash/e}"] & \Topes_n(\om\setminus e)\arrow[ul, phantom,
        "\overset{h}{\ulcorner}"{description, pos=0.25}]
\end{tikzcd}
\]
is a homotopy pushout square.     
\end{theorem}
\begin{proof}
Let $H$ denote the homotopy pushout $\Topes_n(\om/e)\cup^h_{\Topes_n(\om/e) \times \Topes_n(\om|_e)}\Topes_n(\om)$. We have a natural map of posets $h: H \rightarrow \Topes_n(\om \setminus e)$ defined by 
\begin{align*}
    h(g) &= \res_{\backslash/e}(g), \quad \forall g \in\Topes_n(\om/e)  \\
    h(g) &= \res_{E\setminus e}(\mu_e(g)), \quad \forall g \in \Topes_n(\om/e)\times \Topes_n(\om|_e)\\
    h(g) &= \res_{E\setminus e}(g), \quad  \forall g \in \Topes_n(\om).
\end{align*}
To prove that $h$ is a homotopy equivalence, we show that the lower fiber $H\downarrow g \coloneqq h^{-1}(\downarrow g)$ is contractible for all $g = (T, w)\in \Topes_n(\om\setminus e)$. To simplify notation, put $Q = \Topes_n(\om|_e)$. 

\medskip\noindent\emph{Case 1: $T$ is a tope of $\om/e$.}
View $g$ as an element of $\Topes_n(\om/e)$. 
For $f\in \Topes_n(\om)$, let $f_e\in Q$ be its signed weight at $e$.
The map $c:(H\downarrow g) \to (H \downarrow g)$ defined on the three components of $H\downarrow g$ by 
\begin{align*}
 c(f)& =\mu_e(g,f_e),\quad \forall f \in (H\downarrow g) \cap \Topes_n(\om),  \\
 c(f,q) &=(g,q),\quad \forall (f,q) \in (H\downarrow g)\cap (\Topes_n(\om/e) \times Q),\\
 c(f) & = g, \quad \forall f \in (H\downarrow g) \cap \Topes_n(\om/e). 
\end{align*}
is monotone, idempotent, and satisfies $\id\leq c$. The image of $c$ consists of a copy of $Q$ inside $(H\downarrow g)\cap\Topes_n(\om)$ given by the elements of the form $\mu(g,q)$ with $q \in Q$, a copy of $Q$ inside $\Topes_n(\om/e)\times Q$ given by the elements $(g,q)$ with $q \in Q$, together with the single element $g \in \Topes_n(\om/e)$. In this image we have the mixed relations
\[
 (g, q) \leq \mu_e(g,q')\ \Longleftrightarrow\ q\leq q',
\]
On this image, the map $\mu_e(g,q)\mapsto (g,q)$, fixing $(g,q)$ and $g$, is monotone,
idempotent, and pointwise below the identity. Its image has maximum $g$ and so it is contractible.
These two homotopy equivalences prove that $H\downarrow g$ is contractible.

\medskip\noindent\emph{Case 2: $T$ is not a tope of $\om/e$.}
There is a unique sign $\epsilon$ such that $T^\epsilon$, the extension
of $T$ by $\epsilon$ at $e$, is a tope of $\om$. Let $g^{\epsilon, n}$ denote the element of $(H \downarrow g) \cap \Topes_n(\om)$ which is the extension of $g$ by sign $\epsilon$ at $e$ and weight $n$ at $e$. Set also $K = H \downarrow g^{\epsilon, n}$. Finally, consider also the weighted sign $g^{-\epsilon, n+1}$ defined similarly, with sign $-\epsilon$ at $e$, and weight $n+1$ at $e.$ \\

The element $g^{-\epsilon, n+1}$ belongs to $\WTopes_{n+1}(\om).$ The weight vector $T^{-\epsilon}$ is not a tope so it must contain a signed circuit. Choose a circuit $C\preceq T^{-\epsilon}$ minimizing $\mu_{\widehat w}(C)$.
Necessarily $C(e)=-\epsilon$, and $C(s)=T(s)$ on
$\supp C\setminus\{e\}$. Put
\[
 J=(\supp C\setminus\{e\})\cap\{s:w(s)>1\}.
\]
For $S\subseteq J$, define $g_S=(T_S,w_S)$ by flipping signs and lowering
weights by one on $S$, as in Lemma~\ref{lem:circuit-descent-uniform}.
Applying that lemma to $g^{-\epsilon, n+1}$, then restricting to $E\setminus e$, gives
\begin{equation}\label{eq:uniform-gs}
 g_S\in\WTopes_n(\om\setminus e).
\end{equation}
Note that restriction preserves uniformity because every circuit of a deletion
is a circuit of the original matroid. For $S\ne\varnothing$, we can assume that $H\downarrow g_S$ is contractible by induction.

We claim that
\begin{equation}\label{eq:cover}
 H \downarrow g =K\ \cup\ \bigcup_{s\in J}H \downarrow g_{\{s\}}.
\end{equation}
Indeed, if $f\in (H\downarrow g)\cap \Topes_n(\om)$ is not below $g^{\epsilon, n}$, its signed weight at $e$
must be $(-\epsilon,n)$. Its tope agrees with $C$ at $e$, so
circuit--covector orthogonality forces disagreement with $C$ at some
$s\in\supp C\setminus\{e\}$. Similarly, for an element $f$ in $\Topes_n(\om/e)$ or $\Topes_n(\om/e)\times Q$, extend its underlying contraction tope by zero at $e$. Orthogonality with $C$ again forces disagreement at such an $s$. In either situation, if $f = (Y,v)$, then
\[
 Y(s)=-T(s),\qquad v(s)<w(s).
\]
Thus $s\in J$ and $f\leq g_{\{s\}}$, proving \eqref{eq:cover}.

If $J=\varnothing$, this already proves $H \downarrow g =K$. For nonempty $S\subseteq J$, we have the immediate identity
\[
 \bigcap_{s\in S}(H\downarrow g_{\{s\}}) = H\downarrow g_S,
\]
which is contractible by induction. It remains to handle intersections with $K$.

Consider the weighted sign vector $g^{\epsilon, n}_{S}$ on $E$, which is the extension of $g_S$ by sign $\epsilon$ at $e$ and weight $n$ at $e$. The element $g^{\epsilon, n}_{S}$ belongs to $\WTopes_n(\om).$ The intersection $K\cap (H\downarrow g_S)$ only contains elements in $\Topes_n(\om)$ or in $\Topes_n(\om/e) \times Q.$ Fixing the elements in $\Topes_n(\om)$ and sending $p\mapsto\mu(p)$ is a monotone
idempotent map pointwise above the identity. Its image is precisely
\[
 \Topes_n(\om)\cap\downarrow g^{\epsilon, n}_S.
\]
By Lemma~\ref{lem:contractible-fibers-uniform}, this image is contractible, and hence
$K\cap (H\downarrow g_S)$ is contractible.

To conclude, we see that all nonempty finite intersections in the lower-ideal cover
\eqref{eq:cover} are contractible, so by Theorem~\ref{thm:cover-contractible} the lower fiber $H\downarrow g$ is contractible. Applying Theorem~\ref{thm:quillen-type-A} finishes the proof of Theorem~\ref{thm:deletion-contraction-pushout}. 
\end{proof} Let us state some consequences of this theorem. For the rank $1$ matroid $\U_{1,1}$, the poset $\Topes_n(\U_{1,1})$ has the form
\begin{equation}\label{eq:topes(u1,1)}
    \Topes_n(\U_{1,1})=\{(\pm,i)\mid i=1,\ldots,n\}, \quad (s,i)<(t,j) \text{ if } i<j.
\end{equation} Its geometric realisation $|\Topes_n(\U_{1,1})|\cong S^{n-1}$ is homeomorphic to an $(n-1)$-sphere. This leads to the following 
\begin{corollary}[Deletion-contraction cofiber sequence]\label{cor:cofiber}
We have a homotopy equivalence 
\[
\Cofib(|\Topes_n(\om)|\to|\Topes_n(\om/e)|)\simeq \Sigma_+^n|\Topes_n(\om/e)_+|,
\] where $\Cofib$ denotes the homotopy cofiber, $\Sigma_+^nX=\Sigma^n(X_+)$ is an iterated suspension.
\end{corollary}
\begin{proof}
Recall that a homotopy pushout induces an equivalence of cofibers 
\begin{equation}\label{eq:pushout+cofiber}
\begin{tikzcd}[column sep=large]
{|\Topes_n(\om/e)|\times |\Topes_n(\om|_e)|} \arrow[r,"|\pr_1|"] \arrow[d,"\mu_e"]
  & {|\Topes_n(\om/e)|}\arrow[r] \arrow[d]
  & {\Cofib(|\pr_1|)} \arrow[d,"\simeq"]\\
{|\Topes_n(\om)|} \arrow[r,"\res"]
  & {|\Topes_n(\om\setminus e)|} \arrow[ul, phantom,
        "\overset{h}{\ulcorner}"{description, pos=0.15}] \arrow[r]
  & \Cofib(|\res|).
\end{tikzcd}
\end{equation} Since $\om|_e=\U_{1,1}$ and $|\Topes_n(\U_{1,1})|\cong S^{n-1}$, we compute the cofiber $\Cofib(|\pr_1|)$ using the cofiber replacement $|\Topes_n(\om/e)|\times S^{n-1}\hookrightarrow |\Topes_n(\om/e)|\times D^n$ and hence 
    \[ \Cofib(|\Topes_n(\om)|\to|\Topes_n(\om/e)|)\simeq  \Cofib(|\pr_1|) \simeq \frac{|\Topes_n(\om/e)| \times D^n}{|\Topes_n(\om/e)| \times S^{n-1}} \simeq \Sigma^n|\Topes_n(\om/e)|_{+}. 
    \]
\end{proof}
As an immediate consequence, we obtain a deletion-contraction sequence for the cohomology groups of $\Topes_n(\om).$
\begin{corollary}\label{coro:deletion-contraction-sequence}
    Let $\om$ be an oriented matroid and $n\geqslant 2$ an integer. If $e \in E(\om)$ is such that $\{e\}$ is a flat, we have a long exact sequence of cohomology groups
    \[
\cdots \to H^{\bullet}(\Topes_n(\om\setminus e))\xrightarrow{\res} H^{\bullet}(\Topes_n(\om))\overset{\partial_e}{\to} H^{\bullet-(n-1)}(\Topes_n(\om/e))\to H^{\bullet+1}(\Topes_n(\om\setminus e))\to \cdots
\]
\end{corollary}
\begin{proof}
 This follows from the long exact sequence induced by the cofiber sequence in Corollary~\ref{cor:cofiber} and the suspension isomorphism $H^{\bullet}(\Sigma^n_+X)\cong H^{\bullet-n}(X)$   
\end{proof}

In the next section we show that this long exact sequence splits into so-called "deletion-contraction" short exact sequences.

\subsection{Cohomology of $\Topes_n$}\label{sec:generators-of-alg}
Since $\Topes_n$ is an $\OM$-presheaf and cohomology is a contravariant functor, its cohomology $H^{\bullet}(\Topes_n)$ forms an $\OM$-copresheaf. For simplicity, for a weak signed map $\phi\colon \om\to\N$, we denote by $\phi_*$ the corresponding homomorphism of cohomology rings. In this section we establish the following isomorphism.
\begin{theorem}\label{thm::pres_H(Topes)}
For every $n\geqslant 2$, we have an isomorphism of graded algebras
\begin{equation} \label{eq:comparison-cohomology-OS}
\OS_n^\bullet(\om)\cong
H^\bullet(\Topes_n(\om)),
\end{equation}
that is compatible with the $\OM$-copresheaf structure on both sides.
\end{theorem}

The proof splits into several steps. Firstly, we determine generators of $H^{\bullet}(\Topes_n(\om))$. We recall the description of the poset $\Topes_n(\U_{1,1})$ for the rank one matroid $\U_{1,1}$~\eqref{eq:topes(u1,1)}. Since $|\Topes_n(\U_{1,1})|\cong S^{n-1}$, we have $H^{\bullet}(\Topes_n(\U_{1,1})\cong \Z[\omega^{n-1}]/(\omega^2)$, and we consider an explicit representative $\omega^+\in \mathrm{C}^{n-1}(\Topes_n(\U_{1,1}))$ of the generator defined by
\[
\omega^{+}((+,1) < (+,2) < \cdots < (+, n)) = 1
\]
and $\gen(\chain) = 0$ for every other $n$-chain $\chain$ of $\Topes_n(\om)$. For any oriented matroid $\om$ and any element $e\in E(\om)$, we define an element
\[
\omega^{+}_e \coloneqq \iota^*_e(\omega^{+}) \in H^{n-1}(\Topes_n(\om)), \text{ where } \iota_e\colon \U_{1,1}\overset{e}{\to} \om
\]
The following lemma explains the compatibility of these classes with respect to signed weak maps.
\begin{lemma}\label{lem:induced-cohomology-topes} We have
\begin{itemize}
    \item[(i)] For a signed weak map $\phi\colon \om\to \N$, we have
\begin{equation}\label{eq:signed_weak_maps_omega_+}
\phi_*\omega^+_e=\epsilon^n_{\phi}(e)\omega^+_{f_{\phi}(e)}
\end{equation}
    \item[(ii)] For parallel elements $e,e'\in E(\om)$, we have $\omega^+_e=\omega^+_{e'}$. For antiparallel elements, we have $\omega_e=(-1)^n\omega^+_{e'}$
\end{itemize}
\end{lemma}
\begin{proof} Let $\tau\colon \U_{1,1}\to \U_{1,1}$ be the unique involution that changes a sign $\tau(+)=(-)$. The geometric realisation $|\tau|$ defines a homeomorphisms of $|\Topes_n(\U_{1,1})|\cong \Sph^{n-1}$ that homotopic to the antipodal map, which has degree $(-1)^n$; hence we have $\tau_*\omega^+=(-1)^n\omega^+$ in $H^{n-1}(\Topes_n(\U_{1,1}))$.

(i) For a signed weak map $\phi\colon \om\to \N$ and an inclusion $\iota_e\colon \U_{1,1}\to \om$, we have
\[
\phi\iota_e=\begin{cases}
\iota_{f_{\phi}(e)},\text{ if }\epsilon_{\phi}(e)=+
\\
\iota_{f_{\phi}(e)}\tau,\text{ if }\epsilon_{\phi}(e)=-
\end{cases}
\] Combining these identities with $\phi_*\omega^+_e=(\phi\iota_e)_*\omega^+$ implies the desired identity~\eqref{eq:signed_weak_maps_omega_+}.\\

(ii) Thanks to~\eqref{eq:signed_weak_maps_omega_+}, we can reduce ourseleves to the case $\om=\U_{1,2}$ the rank $1$ matroid with two elements $\{1,2\}$. Let $\sigma\colon \U_{1,2}\to \U_{1,2}$ be a transposition of elements, so we have $\omega_2=(12)_*\omega_1$ by~\eqref{eq:signed_weak_maps_omega_+}. If $e,e'$ are parallel, then we have $(12)\omega_1^+=\omega_1^+$ at the cochain level, which implies $\omega^+_1=\omega^+_2$. Otherwise, we have $(12)\omega_1^+=\omega^-=\tau\omega_1^+$ at the cochain level, which implies $\omega^+_2=(-1)^n\omega^+_1$ in cohomology. 
\end{proof}  

Next, let us describe how the connecting homomorphism map $\partial_e$ from Corollary~\ref{coro:deletion-contraction-sequence} acts on $\omega^+$ monomials

\begin{lemma}\label{lem:connecting-gysin}
For a flat $e\in E(\om)$, the connecting homomorphism
\[
\partial_e\colon H^{\bullet}(\Topes_n(\om))
\longrightarrow H^{\bullet-(n-1)}(\Topes_n(\om/e))
\] acts on $\omega^+$ monomials by the rule
\begin{equation}\label{eq:connecting-monomials}
\partial_e(\omega_e^+\omega_S^+)=\omega_S^+,
\qquad
\partial_e(\omega_S^+)=0,
\qquad \forall S\subseteq E(\om)\setminus\{e\}.
\end{equation}
\end{lemma}
\begin{proof} Consider the following cofiber sequence
\[
X\times S^{n-1}\longrightarrow X\times D^n
\longrightarrow \Th_n(X)\cong\Sigma_+^nX.
\]
The connecting map has the form
\[
\partial\colon H^{\bullet}(X)\otimes H^{\bullet}(S^{n-1})
\longrightarrow H^{\bullet-(n-1)}(X), \quad \alpha\otimes \omega\mapsto (-1)^{|\alpha|}\alpha, \quad \alpha\otimes 1\mapsto 0,
\] where $\omega\in H^{n-1}(S^{n-1})$ is a chosen generator. Applying this formula to the map of cofiber sequences~\eqref{eq:pushout+cofiber} and combining with Lemma~\ref{lem:induced-cohomology-topes}, we deduce the required formula.
\end{proof}

We recall the generalised Orlik--Solomon algebras $\OS_n^\bullet(\om)$ from Section~2.4. The $\OM$-copresheaf structure of $\OS_n$ was described in Proposition~\ref{prop:presheaf-Orlik-Solomon}. We first treat free oriented matroids. Let $\U_{k,k}$ denote the free oriented matroid on $k$ elements. Since $\U_{k,k}\cong\U_{1,1}^{\oplus k}$, we have
\[
\Topes_n(\U_{k,k})\cong\Topes_n(\U_{1,1})^{\times k},
\qquad
|\Topes_n(\U_{k,k})|\simeq(S^{n-1})^{\times k}.
\]
Consequently,
\[
H^\bullet(\Topes_n(\U_{k,k}))
\cong\OS_n^\bullet(\U_{k,k})
\cong
\frac{\Z[\omega_1,\ldots,\omega_k]}
{\langle\omega_j^2\mid 1\leqslant j\leqslant k\rangle}.
\]
We next consider oriented matroids whose underlying matroid is $\U_{k-1,k}$.

\begin{lemma}\label{lem:presentation-circuit-matroids}
The presentation in Theorem~\ref{thm::pres_H(Topes)} holds for oriented matroids with underlying matroid $\U_{k-1,k}$, for $k\geqslant 2$.
\end{lemma}
\begin{proof}
We proceed by induction on $k$. The base case $k =2$ is given by Lemma~\ref{lem:induced-cohomology-topes}. After reorientation, we may assume that $C=(+,\ldots,+)$ is a signed circuit, and we write $\U_{k-1,k}$ for this oriented matroid. Its algebra has the presentation
\[
\OS_n^\bullet(\U_{k-1,k})
=
\frac{\Z[\omega_1,\ldots,\omega_k]}
{\left\langle
\sum_{j=1}^k(-1)^{j(n-1)}
\omega_1\cdots\widehat{\omega_j}\cdots\omega_k
\right\rangle
+\langle\omega_j^2\mid 1\leqslant j\leqslant k\rangle}.
\] For the induction step, let $k\geqslant 3$ and choose an element $e\in E$. Then
\[
\U_{k-1,k}/e\cong\U_{k-2,k-1},
\qquad
\U_{k-1,k}\setminus e\cong\U_{k-1,k-1}.
\]
By induction, the cohomology of $\Topes_n(\U_{k-1,k}/e)$ is generated by the classes $\omega_f^+$, for $f\in E\setminus\{e\}$, in degree $n-1$. The generator pullback rule therefore shows that
\[
\mu_e^*\colon H^\bullet(\Topes_n(\U_{k-1,k}))
\longrightarrow
H^\bullet(\Topes_n(\U_{k-2,k-1}))
\otimes H^\bullet(\Topes_n(\U_{1,1}))
\]
is surjective. By Lemma~\ref{lem:connecting-gysin}, so is $\partial_e$. Thus the deletion-contraction long exact sequence of Corollary~\ref{coro:deletion-contraction-sequence} breaks into short exact sequences
\[
0\longrightarrow
\OS_n^q(\U_{k-1,k-1})
\xrightarrow{\ \iota_e\ }
H^q(\Topes_n(\U_{k-1,k}))
\xrightarrow{\ \partial_e\ }
\OS_n^{q-(n-1)}(\U_{k-2,k-1})
\longrightarrow 0.
\]
Here $\iota_e$ is induced by the deletion map, using the already established presentation for the free matroid.

It remains to verify the circuit relation
\[
\partial\omega_C^+
:=
\sum_{j=1}^k(-1)^{j(n-1)}
\omega_{e_1}^+\cdots\widehat{\omega_{e_j}^+}\cdots\omega_{e_k}^+
=0.
\]
By \eqref{eq:connecting-monomials} and the induction hypothesis,
\[
\partial_e(\partial\omega_C^+)
=\pm\partial\omega_{C\setminus e}^+
=0
\quad\text{in }\OS_n^\bullet(\U_{k-2,k-1}).
\]
Exactness therefore gives
\[
\partial\omega_C^+
\in
\iota_e\bigl(\OS_n^{(k-1)(n-1)}(\U_{k-1,k-1})\bigr).
\]
The group on the right is generated by $\omega_{E\setminus\{e\}}^+$, so there is an integer $\alpha_e$ such that
\[
\partial\omega_C^+=\alpha_e\omega_{E\setminus\{e\}}^+.
\]
The same argument gives $\partial_f(\partial\omega_C^+)=0$ for every $f\in E\setminus\{e\}$. Applying $\partial_f$ to the last identity yields
\[
0=\pm\alpha_e\omega_{E\setminus\{e,f\}}^+
\quad\text{in }H^\bullet(\Topes_n(\U_{k-1,k}/f))
\cong\OS_n^\bullet(\U_{k-2,k-1}).
\]
This monomial has infinite order, so $\alpha_e=0$.

The assignments $\omega_f\mapsto\omega_f^+$, for $f\in E$, thus define a homomorphism
\[
\OS_n^\bullet(\U_{k-1,k})
\longrightarrow H^\bullet(\Topes_n(\U_{k-1,k})).
\]
Comparing the algebraic and cohomological deletion-contraction short exact sequences shows that this homomorphism is an isomorphism.
\end{proof}

\begin{lemma}\label{lem:presentation-map}
For every oriented matroid $\om$, the assignment $\omega_e\mapsto\omega_e^+$ defines a homomorphism of graded algebras
\[
\Phi_{\om}\colon\OS_n^\bullet(\om)
\longrightarrow H^\bullet(\Topes_n(\om)).
\]
\end{lemma}
\begin{proof}
The square-zero relations hold because each $\omega_e^+$ is pulled back from the top cohomology class of a sphere. For a signed circuit $C$ with $k$ elements in its support, the restriction $\om|_{\sC}$ has underlying matroid $\U_{k-1,k}$. Its circuit relation holds by Lemma~\ref{lem:presentation-circuit-matroids}, and the same relation holds in $H^\bullet(\Topes_n(\om))$ by Lemma~\ref{lem:induced-cohomology-topes} as well.
\end{proof}

\begin{proof}[Proof of Theorem~\ref{thm::pres_H(Topes)}]
We argue by induction on the cardinality $k=|E(\om)|$, using deletion and contraction. Notice that if $e$ and $e'$ are two parallel or antiparallel elements of $\om$, then removing $e$ does not change either side of~\eqref{eq:comparison-cohomology-OS}. Hence, we can assume that there exists $e\in E(\om)$ such that $\{e\}$ is a flat. By induction,
\[
H^\bullet(\Topes_n(\om/e))\cong\OS_n^\bullet(\om/e)
\]
is generated by the classes $\omega_f^+$. As in the proof of Lemma~\ref{lem:presentation-circuit-matroids}, the map $\mu_e^*$ is surjective, and Lemma~\ref{lem:connecting-gysin} implies that $\partial_e$ is surjective. The cohomological deletion-contraction sequence therefore breaks into short exact sequences. Comparing it with the algebraic deletion-contraction sequence gives the commutative diagram
\[
\begin{tikzcd}[column sep=small]
0 \arrow[r]
  & H^q(\Topes_n(\om\setminus e)) \arrow[r]
  & H^q(\Topes_n(\om)) \arrow[r,"\partial_e"]
  & H^{q-(n-1)}(\Topes_n(\om/e)) \arrow[r]
  & 0\\
0 \arrow[r]
  & \OS_n^q(\om\setminus e) \arrow[r] \arrow[u,"\cong"]
  & \OS_n^q(\om) \arrow[r] \arrow[u,"\Phi_{\om}"]
  & \OS_n^{q-(n-1)}(\om/e) \arrow[r] \arrow[u,"\cong"]
  & 0.
\end{tikzcd}
\]
The outer vertical maps are isomorphisms by induction, so the middle map is also an isomorphism. \\

Compatibility with the copresheaf structure follows from the definition of $\Phi_{\om}$ on generators described in Lemma~\ref{lem:induced-cohomology-topes}. 

\end{proof}
\subsection{Applications}\label{sec:applications-pushout}

In this section we give three applications of the results of the previous sections. The first application is a generalization of a result of Schaper~\cite[Corollary 3]{Schaper1997}.
\begin{corollary}[Stable decomposition]
For every oriented matroid $\om$, we have the wedge decomposition
\[
\Sigma|\Topes_n(\om)|\simeq \bigvee^{\rk\om}_{i=1} \bigvee_{j=1}^{b_i} S^{i(n-1)+1},
\]
where $b_i = \dim H^{(n-1)i}(\Topes_n(\om))$ for all $i\geqslant 1$.
\end{corollary}
\begin{proof}
Write $E= E(\om).$ The canonical weak map (Example~\ref{ex:canonical-weak-map}) induces a map
\[
    j\colon|\Topes_n(\om)| \longrightarrow|\Topes_n(\U_{E})| \simeq(S^{n-1})^{{E}}.
\]
This induces 
\[
    \Sigma j\colon \Sigma|\Topes_n(\om)| \longrightarrow \Sigma|\Topes_n(\U_{E})| \simeq \bigvee_{\emptyset \neq I \subseteq E}(S^{(n-1)|I|+ 1}).
\]
Recall that for all $n\geqslant 2$, the generalized Orlik-Solomon algebra $\OS_n^{\bullet}(\om)$ has a $\Z$-basis consisting of all the monomials $\omega_I$ where $I$ is an NBC basis (see~\cite[Theorem~3.43]{OrlikTerao1992} for the case $n=2$; the proof works for any $n$). We can compose $\Sigma j$ with the map that collapses of every sphere indexed by a set $I$ which is not an NBC basis, to obtain a map 
\[
    \Sigma |\Topes_n(\om)|\rightarrow \bigvee_{\substack{\emptyset \neq I \subseteq E \\ I \textrm{ NBC basis}}}(S^{(n-1)|I|+ 1}).
\]
By Theorem~\ref{thm::pres_H(Topes)}, this map is an isomorphism in integral cohomology, so it is an isomorphism in integral homology. Moreover, the two spaces are simply-connected, so by the Whitehead theorem~\cite[Corollary 4.33]{Hatcher2002}, this map is a homotopy equivalence. 
\end{proof}

The next consequence concerns fundamental groups and generalizes known results for hyperplane arrangement complements (see, for instance, \cite[Proposition 5.6.1]{ShimadaZariskiVanKampen}). 
\begin{corollary}[Fundamental groups]\label{prop::fund_grps}
Let $\om$ be an oriented matroid.
\begin{itemize} 
    \item[(i)] For $n=2$, if $e\in E(\om)$ is such that $\{e\}$ is a flat, then we have a short exact sequence of groups
\[
1\to \langle\langle t_e\rangle\rangle\to \pi_1(\Topes_2(\om))\to \pi_1(\Topes_n(\om\setminus e))\to 1,
\] where $\langle\langle t_e\rangle\rangle$ is a normal closure and $t_e$ comes from $t_e=\mu_e(X_0,-)\colon S^1\cong \Topes_n(\om|_e)\to \Topes_n(\om)$.
\item[(ii)] For $n\geq 2$, the poset $\Topes_n(\om)$ is $(n-2)$-connected
\end{itemize}
\end{corollary}
\begin{proof}
Recall that, by Van Kampen theorem, the fundamental group of a homotopy pushout is a pushout of fundamental groups. Hence we have
\[
\pi_1(\Topes_n(\om\setminus e))\cong \mathrm{colim}(\pi_1(\Topes_n(\om/e))\leftarrow \pi_1(\Topes_n(\om/e))\times\pi_1(\Topes_n(\om|_e))\rightarrow \pi_1(\Topes_n(\om)))
\] Recall that $\Topes_n(\om|_e)\cong S^{d-1}$, which is simply-connected for $n>2$; hence $\pi_1(\Topes_n(\om\setminus e))\cong \pi_1(\Topes_n(\om))$. By induction on the cardinality of $E(\om)$ and using $\Topes_n(\U_{r,r})\cong (S^{n-1})^{r}$ which is simply-connected, we deduce that $\Topes_n(\om)$ is simply-connected. Since $H^{\bullet}(\Topes_n)$ are torsion free and zero for $\bullet<n-1$, the Hurewicz theorem implies that $\Topes_n(\om)$ is $(n-2)$-connected.
\end{proof} 

Finally, we obtain the following generalization of Randell's vanishing result~\cite{randell1997homotopy}.
\begin{corollary}[Vanishing result]
For every oriented matroid $\om$ and every $n\geqslant 2$, the Hurewicz map
\[
\pi_q(|\Topes_n(\om)|)
\longrightarrow H_q(|\Topes_n(\om)|;\Z)
\]
is zero for $q\geqslant n$.
\end{corollary}
\begin{proof}
Denote $E = E(\om)$. The canonical weak map (Example~\ref{ex:canonical-weak-map}) induces a map
\[
j\colon|\Topes_n(\om)|
\longrightarrow|\Topes_n(\U_{E})|
\simeq(S^{n-1})^{{E}}.
\]
By Theorem~\ref{thm::pres_H(Topes)}, the induced cohomology map is a surjection
\[
\OS_n^\bullet(\U_{E})\twoheadrightarrow\OS_n^\bullet(\om).
\]
The generalized Orlik--Solomon algebras are free abelian, so the integral homology groups are also free abelian. The universal coefficient theorem then shows that $j_*$ is injective.

For $q\geqslant n$, the Hurewicz map of $(S^{n-1})^{E}$ vanishes: its homotopy groups are products of those of the sphere factors, and the Hurewicz map of each factor lands in $H_q(S^{n-1};\Z)=0$. Naturality of the Hurewicz map and injectivity of $j_*$ give the result.
\end{proof}

\section{Deligne groupoids and Fundamental groups}\label{sec:deligne}

Thanks to Proposition~\ref{prop::fund_grps}, higher complement arrangements are simply-connected for $n>2$. In the two-dimensional case, for an oriented matroid $\om$, the so-called Deligne groupoid $\Del(\om)$ defined in terms of tope graphs forms a combinatorial model of the fundamental groupoid $\Pi_1(\Sal(\om))$. These groupoids were first introduced by Deligne~\cite{deligne1972immeubles} for simplicial arrangements, and were later
formulated by Paris for arbitrary hyperplane arrangements~\cite{paris2000fundamental}. See~\cite{delucchi2006combinatorial} for a modern approach in the setting of oriented matroids.

We define an $\OM$-presheaf structure on the Deligne groupoids $\Del(\om)$ and show that the resulting presheaf is equivalent to the fundamental groupoid $\OM$-presheaf $\Pi_1(\Topes_2)$. Using this model, we recognize a categorical circle action of $\B\Z\simeq S^1$ on $\Del$ that corresponds to the action of $\Cmplx^*$ on the complements of complex hyperplane arrangements $\U_{\Cmplx}(\A)$ in the realizable case. 
\subsection{Tope graph and Deligne positive categories} 
For an oriented matroid $\om$, its tope graph is an undirected graph $\Gr_{\om}$ on the vertex set $\Topes(\om)$ with an edge between $X$ and $Y$ when the separator $S(X,Y)$ is a rank $1$ flat (when $\om$ is simple, it means $|S(X,Y)|=1$). Geometrically, the corresponding regions share a facet.

A path $\gamma\colon X\to Y$ in $\Gr_{\om}$ is a \textit{geodesic} if it is of minimal lenghth among all paths from $X$ to $Y$. Every geodesic crosses each parallel class in $S(X,Y)$ exactly once; in particular, $l(\gamma)=|S(X,Y)|$ when $\om$ is simple~\cite[Prop~4.2.3]{bjorner1999oriented}. Hence, two geodesics $\gamma\colon X\to Y, \eta\colon Y\to Z$ concatenate to a geodesic precisely when 
\[
S(X,Z)=S(X,Y)\coprod S(Y,Z).
\]

We recall the construction of the positive Deligne category $\Del^+(\om)$. Let $\vec\Gr_{\om}$ be the directed graph obtained by replacing each edge of $\Gr_{\om}$ with two opposite directed arrows. The \textit{positive Deligne category} $\Del^+(\om)$ is the category generated by this directed graph modulo geodesic relations: any two geodesics with the same endpoints are identified.

\begin{figure}[ht]
    \centering
    \begin{tikzpicture}[
  vertex/.style={
    circle, fill=black, inner sep=0pt,
    minimum size=4pt, outer sep=1pt
  },
  directed/.style={
    -{Stealth[length=2.1mm,width=1.5mm]},
    shorten <=1pt, shorten >=1pt
  },
  reverse/.style={directed, draw=black!50, line width=0.55pt},
  gammapath/.style={directed, draw=blue!70!black, line width=1.05pt},
  etapath/.style={directed, draw=red!70!black, line width=1.05pt},
  every label/.style={font=\small, label distance=5pt}, scale=0.7
]
  \foreach \i/\ang/\signs in {
    0/180/{+,+,+,+},
    1/135/{-,+,+,+},
    2/90/{-,-,+,+},
    3/45/{-,-,-,+},
    4/0/{-,-,-,-},
    5/-45/{+,-,-,-},
    6/-90/{+,+,-,-},
    7/-135/{+,+,+,-}
  }{
    \node[vertex, label={\ang:{$(\signs)$}}]
      (v\i) at (\ang:3.5cm) {};
  }
  \foreach \a/\b in {0/1,1/2,2/3,3/4}{
    \draw[reverse]   (v\b) to[bend left=25] (v\a); 
    \draw[gammapath] (v\a) to[bend left=30] (v\b); 
  }
  \foreach \a/\b in {0/7,7/6,6/5,5/4}{
    \draw[reverse] (v\b) to[bend left=30] (v\a); 
    \draw[etapath] (v\a) to[bend left=25] (v\b); 
  }
  \node[text=blue!70!black, font=\Large] at (0, 2.25) {$\gamma$};
  \node[text=red!70!black,  font=\Large] at (0,-2.25) {$\eta$};
  
  \node at (0,0) {$[\gamma]=[\eta]$};
\end{tikzpicture}
    \caption{$\Del^+(\U_{2,4})$ and an example of a relation}
    \label{fig:deligne}
\end{figure}
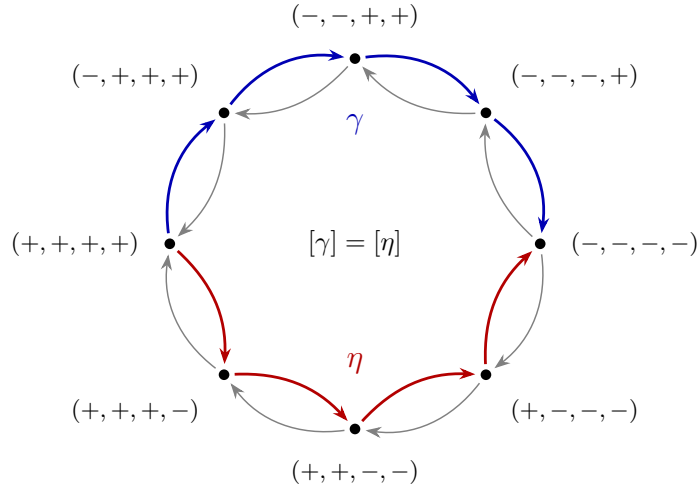

We denote by $\gamma_{XY}\colon X\to Y$ the common class of geodesics from $X$ to $Y$. Equivalently, these classes generate $\Del^+(\om)$ modulo the relation
\begin{equation}\label{eq::geodesic_rel}
    \gamma_{YZ}\gamma_{XY}=\gamma_{XZ}, \text{ if }S(X,Z)=S(X,Y)\coprod S(Y,Z)
\end{equation}

\begin{prop}
Positive Deligne categories form $\OM$-presheaf in the category $\mathsf{Cat}$ of small categories 
\[
\Del^+\colon\OM^{\op}\to \mathsf{Cat}
\]
\end{prop}
\begin{proof}
For a morphism of matroids $\phi\colon \om\to \N$, we define a functor
\[
\Del(\phi)\colon\Del^+(\N)\to \Del^+(\om), \quad X\mapsto \phi^*X, \quad \gamma_{XY}\mapsto \gamma_{\phi^*X\phi^*Y}
\] 
The identity $S(\phi^*X,\phi^*Y)=f^{-1}_{\phi}S(X,Y)$ shows that this correspondence $\gamma_{XY}\mapsto \gamma_{\phi^*X\phi^*Y}$ preserves~\eqref{eq::geodesic_rel}. Since morphisms are generated by geodesics, $\Del(\phi)$ is a well-defined functor.
\end{proof}

\subsection{Deligne groupoid and Fundamental groups} The \textit{Deligne groupoid} $\Del(\om)$ is the groupoid completion
\[
\Del(\om):=\Del^+(\om)[\mathrm{Mor}^{-1}].
\]
In other words, we invert all morphisms in $\Del^+(\om)$. The images in $\Del(\om)$ of morphisms from $\Del^+(\om)$ are called \textit{positive}. By the universal property of groupoid completions, the $\OM$-presheaf structure on $\Del^+$ lifts to an $\OM$-presheaf structure on $\Del$ with values in groupoids. Let us explain the relation of this $\OM$-presheaf to fundamental groups of $2$-dimensional complements $\U_2(\A)$. 

Given a topological $\OM$-presheaf $\O\colon \OM^{\op}\to \Top$, composing it with the fundamental group construction $\pi_1(-)$ does not produce an $\OM$-presheaf because there is no canonical choice of basepoints. Fundamental groupoids make this construction functorial without requiring a choice of basepoints. Recall that, for a topological space $X$, its fundamental groupoid $\Pi_1(X)$ is a groupoid whose objects are points of $X$ and morphisms are homotopy classes of paths. In particular, for a point $x_0\in X$, we have an isomorphism of groups
\[
\End_{\Pi_1(X)}(x_0)\cong \pi_1(X,x_0)
\] In the special when $X=\Po$ is a poset with Alexandrov topology, its fundamental groupoid $\Pi_1(\Po)$ is equivalent to the groupoid completion of $\Po$, viewed as a category, with respect to all morphisms. Explicitly, $\Pi_1(\Po)$ is generated by $[x\leq y]$ and their formal inverse $[x\leq y]^{-1}$ modulo the relation $[y\leq z][x\leq y]=[x\leq z]$.

\begin{theorem}
 We have an equivalence of $\OM$-presheaves in groupoids
\[
\Del\simeq \Pi_1(\Topes_2)
\]  
\end{theorem}
\begin{proof}

We construct a morphism of presheaves $\Psi\colon \Pi_1(\Topes_2)\to \Del$ by the rule
\[
\Psi_{\om}((X,w_X))=X, \quad \Psi_{\om}((X,w_X)\leq (Y,w_Y))=\gamma_{XY}.
\]
\begin{claim*} $\Psi_{\om}\colon \Pi_1(\Topes_2(\om))\to \Del(\om)$ is a well-defined functor
\end{claim*}
\begin{proof}
We have to check that, for a triple $(X,w_X)\leq (Y,w_Y)\leq (Z,w_Z)$, we have  $\Psi((Y,w_Y)\leq (Z,w_Z))\Psi((X,w_X)\leq (Y,w_Y))=\gamma_{YZ}\gamma_{XY}=\gamma_{XZ}=\Psi((X,w_X)\leq (Z,w_Z))$ or, equivalently, $S(X,Z)=S(X,Y)\coprod S(Y,Z)$. Since $(X,w_X)\leq (Y,w_Y)\leq (Z,w_Z)$, the weighting $w_Y$ takes values $2$ on $S(X,Y)$ and $1$ on $S(Y,Z)$, hence $S(X,Z)$ and $S(Y,Z)$ are disjoint, which implies the desired identity. 
\end{proof} $\Psi$ is functorial with respect to signed weak maps thanks to the computation below
\[
\Psi_{\om}(\phi^*[(X,w_X)\leq (Y,w_Y)])=\Psi_{\om}([(\phi^*X,\phi^*w_X)\leq (\phi^*Y,\phi^*w_Y)])=\gamma_{\phi^*X\phi^*Y}=\phi^*\gamma_{XY}=\phi^*\Psi_{\N}([(X,w_X)\leq (Y,w_Y)]),
\] so $\Psi\colon \Pi_1(\Topes_2)\to \Del$ is a morphism of groupoid $\OM$-presheaves.

We define an inverse functor $\Phi\colon \Del\to \Pi_1(\Topes_2)$ as follows. For a pair of topes $X,Y$, let $w_{XY}\colon E(\om)\to [2]$ be $2$ on the separating set $S(X,Y)$ and $1$ elsewhere. We have $(X,1)\leq(X,w_{XY})\geq(Y,1)$, so it produces a path $\alpha_{XY}:=[(X,w_{XY})\geq(Y,1)]^{-1}[(X,1)\leq(X,w_{XY})]$ in $\Pi_1(\Topes_2)$.
\begin{claim*}
The correspondence $X\mapsto (X,1)$ and $\gamma_{XY}\mapsto \alpha_{XY}$ defines a functor
\[
\Phi_{\om}\colon \Del(\om)\to \Pi_1(\Topes_2)
\]
\end{claim*}
\begin{proof} We have to check that, for a triple $X,Y,Z$ with $S(X,Y)\coprod S(Y,Z)=S(X,Z)$, we have $\alpha_{YZ}\alpha_{XY}=\alpha_{XZ}$. By construction, we have $(X,w_{XY})\leq (X,w_{XZ})$ and $(Y,w_{YZ})\leq (X,w_{XZ})$, hence both paths $\alpha_{YZ}\alpha_{XY}$ and $\alpha_{XZ}$ are contained in the order complex $\Delta((\Topes_2)_{\leq (X,w_{XZ})})$.  This subposet has a maximal element, hence this subcomplex is contractible, so the paths $\alpha_{YZ}\alpha_{XY}$ and $\alpha_{XZ}$ are homotopy equivalent. 
\end{proof} 

For a signed weak map $\phi\colon \om\to \N$, since $f^*_{\phi}w_{XY}=w_{\phi^*X\phi^*Y}$, we have $\phi^*\alpha_{XY}=\alpha_{\phi^*X\phi^*Y}$, which implies functoriality
\[
\Phi(\phi^*\alpha_{XY})=\Phi(\alpha_{\phi^*X\phi^*Y})=\gamma_{\phi^*X\phi^*Y}=\phi^*\gamma_{XY}=\phi^*\Phi(\alpha_{XY})
\] that implies that $\Phi$ is indeed a morphism of $\OM$-presheaves $\Phi\colon \Del\to \Pi_1(\Topes_2)$.

Note that the composition $\Psi\Phi$ is identity $\Id_{\Del}$. On the other hand, for every weighted tope, a morphism $(X,1)\leq (X,w)$, defines a natural transformation $\Phi\Psi\Rightarrow \Id_{\Pi_1(\Topes_2)}$, so $\Psi$ and $\Phi$ defines an equivalence.
\end{proof}

In particular, for a tope $X$ of $\om$, we have an isomorphism of groups
\[
\End_{\Del(\om)}(X)\cong \pi_1(\Topes_2(\om),X)\cong \pi_1(\U_2(\A)),
\] where $\A$ is a pseudohyperplane realisation of $\om$.

\subsection{Circle action and Natural transformations}

Recall that, for a complex hyperplane arrangement $\A$, the projectivisation $\U_{\Cmplx}(\A)\overset{\Cmplx^*}{\to} \Pro(\U_{\Cmplx}(\A))$ is a trivial fibration that leads to a decomposition $\U_2(\A)\cong \Cmplx^{\times}\times \Pro(\U_2(\A))\simeq S^1\times \Pro(\U_2(\A))$. In~\cite{DelucchiFalk2017}, the authors, using the model $\DF$, deduce a similar (non-canonical) decomposition for any oriented matroid. In this section, we recognise a categorical circle action of $S^1\simeq \B\Z$ on the groupoid presheaf $\Del$.

For a category $\C$, its center is a commutative monoid of natural transformations of the identity functor
\[
\Cen(\C):=\mathrm{Nat}(\Id_{\C},\Id_{\C}).
\]Its monoid category $\B\mathcal{Z}(\C)$ acts on $\C$ by
\[
\B\mathcal{Z}(\C)\times \C\to \C, \quad (*,x)\mapsto x, \quad (\alpha,f:x\to y)\mapsto \alpha_yf=f\alpha_x 
\] 

In a similar way, for an $\OM$-presheaf $\O$  with values in categories/groupoids, we define its center $\Cen(\O)$ as a monoid of natural transformations $\mathrm{Nat}_{\mathrm{PSh}_{\OM}}(\Id_{\O},\Id_{\O})$. In other words, we have a collection of natural transformations $\alpha(\om)\in \Cen(\O(\om))$ compatible with pullbacks and direct sums 
\[
\phi^*\alpha(\N)_{X}=\alpha(\om)_{\phi^*X}, \quad \alpha(\om\oplus \N)_{(X, Y)}=(\alpha(\om)_X,\alpha(\N)_Y)
\]
\begin{prop}
We have isomorphisms of commutative monoids 
\[
\mathcal{Z}(\Del^+)\cong \mathbb{N}, \quad \mathcal{Z}(\Del)\cong \Z
\]
\end{prop}
\begin{proof}
Recall that $\Del^+(\U_{1,1})$ is a free category on the directed graph $\begin{tikzcd}
	{+} & {-}
	\arrow[curve={height=-6pt}, from=1-1, to=1-2]
	\arrow[curve={height=-6pt}, from=1-2, to=1-1]
\end{tikzcd}$. We see that the center of $\Del^+(\U_{1,1})$ coincides with the monoid of natural numbers $\mathbb{N}$ with a generator $z=(z_+,z_-)$ given by
\[
z_+=\tau_{-+}\tau_{+-}\in \End_{\Del^+}(+)\cong \mathbb{N}, \quad z_-=\tau_{+-}\tau_{-+}\in \End_{\Del^+}(-)\cong \mathbb{N}
\] So, we want to check that the restriction map
\begin{equation}\label{eq::center_restriction}
 \res\colon \mathcal{Z}(\Del^+)\hookrightarrow \mathcal{Z}(\Del^+(\U_{1,1}))\cong \mathbb{N}, \quad \alpha\mapsto \alpha(\U_{1,1})   
\end{equation}
 is an isomorphism.

Let $\alpha\in \Cen(\Del^+)$. Note that $\alpha$ is completely determined by its $\U_{1,1}$ components $\alpha(\U_{1,1})=(\alpha_+,\alpha_-)$. Indeed, for a matroid $\om$ and a tope $X\in \Topes(\om)\cong \Hom_{\OM}(\om;\U_{1,1})$, we have $X=X^*(+)\in \Del^+(\om)$, so $\alpha_X=X^*(\alpha_+)$. Hence, homomorphism~\eqref{eq::center_restriction} is injective.

To establish an isomorphism of~\eqref{eq::center_restriction}, it remains to check that $z=(z_+,z_-)\in \Cen(\Del^+(\U_{1,1}))$ lift to a natural transformation of $\Id_{\Del^+}$. For a tope $X$ of $\om$, we consider an endomorphism
\[
z_X:=\gamma_{-X,X}\gamma_{X,-X}=X^*z_+\in \End_{\Del^+(\om)}(X)
\] For a morphism $\phi\colon \om\to \Nmat$, thanks to the commutative triangle below, we have
\[
\begin{tikzcd}
	\om && \Nmat \\
	& {\U_{1,1}}
	\arrow["\phi", from=1-1, to=1-3]
	\arrow["{\phi^*X}"', from=1-1, to=2-2]
	\arrow["X", from=1-3, to=2-2]
\end{tikzcd} \Rightarrow \phi^*z_{X}=\phi^*X^*(z_+)=(\phi^*X)^*(z_+)=z_{\phi^*X}
\] It suffices to check that, for a matroid $\om$, the collection $\{z_X\}_{X\in \Topes(\om)}$ forms a natural transformation of $\Id_{\Del^+(\om)}$. For a generating morphism $\tau_{e,X}\colon X\to X_{-e}$, we have to check that $\tau_{e,X} z_X=\tau_{e,X} z_{X_{-e}}$. It follows from functoriality with respect to $\triangle_e\colon \om\to \om/e\oplus\om|_e$ 
\begin{multline*}
\tau_{X,X_{-e}} z_X=\triangle^*_e(\id_{X\setminus e},\tau_{X(e),-X(e)})\triangle_e^*(z_{X\setminus e},z_{X(e)})=\triangle_e^*(z_{X\setminus e},\tau_{X(e),-X(e)}z_{X(e)})=\\=\triangle_e^*(z_{X\setminus e},z_{-X(e)}\tau_{X(e),-X(e)})=z_{X_{-e}}\tau_{X,X_{-e}}    
\end{multline*}

 For Deligne groupod $\OM$-presheaf $\Del$, the proof is the same except that $z$ becomes an invertible natural transformation.
\end{proof}

In particular, for an oriented matroid $\om$ and a tope $X$, we have a canonical central element $z_X\in \pi_1(\Topes_2(\om),X)$. We show that this central element splits.

Under the identification
\[
H^1(\Topes_2(\om))\cong \Hom(H_1(\Topes_2(\om)),\Z)\cong \Hom(\pi_1(\Topes_2);\Z),
\] a generator $\omega_e\in H^1(\Topes_2)$ defines a group homomorphism $\omega_e\colon \pi_1(\Topes_2(\om),X)\to \Z$. It has an explicit combinatorial description
\begin{claim*}
For a positive based loop $\sigma: X=X_0\to X_1\to\cdots X_n=X$ at $X$, we have
\begin{equation}\label{eq::formula_omega_e}
\omega_e(\sigma)=\frac{1}{2}|\{i| e\in S(X_i,X_{i+1})\}|
\end{equation}
\end{claim*}
\begin{proof}
For $\U_{1,1}$, the identity is clear. For a general $\om$, recall that $\omega_e=(\iota_e)_*\omega_+$, where $\iota_e\colon \om|_e\subset \om$. Hence, we have
\[
\omega_e(\sigma)=\omega_+(\iota^*_e\sigma)=\omega_+((\gamma_{-+}\gamma_{+-})^{k_e})=k_e,
\] where $k_e$ is the right-hand side of~\eqref{eq::formula_omega_e}.
\end{proof}
In particular, using $z_X=\gamma_{-X,X}\gamma_{X,-X}$ and formula~\eqref{eq::formula_omega_e}, we have
\[
\forall e\in E(\om)\colon \omega_e(z_X)=1, 
\] which leads us to the decomposition~\cite[Cor~4.10]{DelucchiFalk2017}

\begin{prop}[Central Splitting]
The central subgroup $\langle z_X\rangle \cong \Z \subseteq \pi_1(\Topes_2(\om),X)$ is a direct factor. In other words, up to a non-canonical choice of $e\in E(\om)$, we have the splitting
\[
\pi_1(\Topes_2(\om),X)\cong \Z\times \pi_1(\Topes_2(\om),X)/\langle z_X\rangle
\]
\end{prop}
\bibliographystyle{alpha}
\bibliography{biblio}

\end{document}